\documentclass[12pt, oneside]{scrartcl}
\usepackage[english]{babel}
\usepackage[T1]{fontenc}
\usepackage{makecell}
\usepackage{booktabs}
\usepackage{amssymb}
\usepackage{amsmath}
\usepackage{hhline}
\usepackage{longtable}
\usepackage{amscd}
\usepackage{theorem}
\usepackage{array}
\usepackage{delarray}
\usepackage{multicol}
\usepackage{makeidx}
\usepackage{multirow}
\usepackage{hyperref}

\usepackage{float}

\usepackage{tikz}
\usepackage{calc}
\usetikzlibrary{angles,quotes}
\usetikzlibrary{calc}
\usetikzlibrary{patterns}
\usetikzlibrary{graphs}
\usetikzlibrary{graphs.standard}
\usetikzlibrary{decorations.markings,shapes.geometric,positioning}
\usetikzlibrary{arrows.meta}

\tikzset{
  vtx/.style={circle,draw,fill=white,inner sep=1.6pt,minimum size=13pt},
  blu/.style={blue!80!black,font=\small},
  redt/.style={red!80!black,font=\small}
}
\newtheorem{theorem}{Theorem}

\newtheorem{corollary}{Corollary}
\newtheorem{conjecture}{Conjecture}

\newtheorem{lemma}{Lemma}
\newtheorem{remark}{Remark}[theorem]
\newtheorem{claim}{Claim}[theorem]
\newtheorem{assertion}{Assertion}[theorem]

\newtheorem{problem}{Problem}
\newtheorem{proposition}{Proposition}

\newenvironment{proof}[1][Proof]{\textbf{#1.} }{\ \rule{0.5em}{0.5em}}

\usepackage{tikz}
\author{Ayman El Zein\footnote{Corresponding author. Computer Science Department, University of Sciences and Arts in Lebanon, Beirut, Lebanon; KALMA Laboratory, Faculty of Sciences, Lebanese University, Beirut, Lebanon. email: a.elzein@usal.edu.lb} \footnotemark[3] \and Maidoun Mortada\footnote{KALMA Laboratory, Faculty of Sciences, Lebanese University, Baalbek, Lebanon; Basic and Applied Sciences Research, Al Maaref University, Beirut, Lebanon; and Graph Theory and Operation Research, Department of Mathematics and Physics, Lebanese International University (LIU), Beirut, Lebanon. email: maydoun.mortada@liu.edu.lb} \footnote{The two authors contributed equally and are considered co-first authors.}}

\begin{document} 

\title{Impact of local girth on the $S$-packing coloring of $k$-saturated subcubic graphs}
\maketitle

\begin{abstract}
For a non-decreasing sequence $S=(s_1,s_2,\dots,s_k)$, an $S$-packing coloring of a graph $G$ is a vertex coloring using the colors $s_1,s_2,\dots,s_k$ such that any two vertices assigned the same color $s_i$ are at distance greater than $s_i$. 
A subcubic graph is said to be $k$-saturated, for $0\le k\le3$, if every vertex of degree~3 is adjacent to at most $k$ vertices of degree~3. 
The \emph{local girth} of a vertex is the length of the smallest cycle containing it. 
Bre\v{s}ar, Kuenzel, and Rall [\textit{Discrete Math.} 348(8) (2025),~114477] proved that every claw-free cubic graph is $(1,1,2,2)$-packing colorable, confirming the conjecture for this family. 
Equivalently, a claw-free cubic graph is one in which each $3$-vertex has local girth~3. 
Motivated by this observation and by recent progress on $S$-packing colorings of $k$-saturated subcubic graphs, we study the influence of local girth on their $S$-packing colorability. 
We establish a series of results describing how the parameters of saturation and local girth jointly determine the admissible $S$-packing sequences. 
Sharpness is verified through explicit constructions, and several open problems are posed to delineate the remaining cases.
\end{abstract}

\section{Introduction}

All graphs considered in this paper are finite, simple, and undirected. For a graph $G$, we denote by $V(G)$ and $E(G)$ its vertex and edge sets, respectively. For a vertex $x\in V(G)$, we denote by $N(x)$ the set of the neighbors of $x$ and by $N_2(x)$ the set of vertices at distance exactly two from $x$. For a subset $X \subseteq V(G)$, we write $N[X] = X \cup N(X)$ for the \emph{closed neighborhood} of $X$. The degree of a vertex $x$ is denoted by $d_G(x)$ (or simply $d(x)$), and we write $\Delta(G)$ and $\delta(G)$ for the maximum and minimum degrees of $G$. A vertex $v$ with $d(v)=k$ is called a \emph{$k$-vertex}. A graph $G$ is said to be \emph{subcubic} if $\Delta(G)\le 3$ and \emph{cubic} if $d_G(v)=3$ for every vertex $v$. A subcubic graph  is said to be \emph{$k$-saturated}, for $0 \le k \le 3$, if every vertex of degree~3 is adjacent to at most $k$ vertices of degree~3. Equivalently, a $3$-vertex in a $k$-saturated graph has at least $3-k$ neighbors of degree~2.  In a 3-saturated subcubic graph, we call a 3-vertex \emph{heavy} if all its neighbors are 3-vertices. A $3$-saturated subcubic graph  is said to be \emph{$(3,k)$-saturated}, for $0\leq k\leq 3$, if  each heavy vertex is adjacent to at most $k$ heavy vertices. This classification, introduced in~\cite{12,13}, provides a natural hierarchy of subcubic graphs according to the density of adjacent $3$-vertices and will serve as one of the structural parameters in our analysis. A \emph{chord} in a cycle is an edge joining two vertices that are not adjacent on the cycle. 

Given a non-decreasing sequence of positive integers $S=(s_1,s_2,\ldots ,s_k)$, an \emph{$S$-packing coloring} of a graph $G$ is a partition of $V(G)$ into subsets $V_1,V_2,\ldots ,V_k$ such that for any two distinct vertices $u,v\in V_i$, we have $dist_G(u,v)>s_i$. 
The concept, introduced by Goddard \emph{et al.}~\cite{7}, has since inspired intensive research, particularly on graphs of maximum degree three (see the survey~\cite{2}). 
The smallest integer $k$ such that $G$ admits a $(1,2,\ldots ,k)$-packing coloring is called the \emph{packing chromatic number} of $G$, denoted $\chi_\rho(G)$. 
This parameter was introduced by Goddard \emph{et al.}~\cite{7} under the name \emph{broadcast chromatic number}, who showed that deciding whether $\chi_{\rho}(G)\le4$ is NP-hard. The study of bounding $\chi_{\rho}(G)$ and $\chi_{\rho}(S(G))$, where $S(G)$ denotes the graph obtained from $G$ by subdividing every edge, has been the subject of many works~\cite{1,3,4,6,9,16}. The question of whether $\chi_{\rho}$ is bounded for subcubic graphs was answered negatively by Balogh, Kostochka, and Liu~\cite{1}, with explicit unbounded constructions later provided in~\cite{2}. 
However, Balogh, Kostochka, and Liu~\cite{1} proved that $\chi_{\rho}(S(G)) \le 8$ for every subcubic graph $G$, while Gastineau and Togni~\cite{6} raised the question of whether $\chi_{\rho}(S(G)) \le 5$. This conjecture was later supported by Bre\v{s}ar \emph{et al.}~\cite{4}, who proposed that $\chi_{\rho}(S(G)) \le 5$ holds for all subcubic graphs. 
Moreover, Gastineau and Togni~\cite{16} observed that if a subcubic graph $G$ is $(1,1,2,2)$-packing colorable, then it satisfies $\chi_{\rho}(S(G)) \le 5$.

Several significant results have been obtained on the $S$-packing coloring of subcubic graphs. It was recently shown in \cite{11} that every subcubic graph $G$ admits a $(1,1,2,2,3)$-packing coloring, which implies the general upper bound $\chi_{\rho}(S(G)) \le 6$ for this class. 
Later, El~Zein and Mortada~\cite{AM1} proved that every subcubic graph with minimum degree less than three is $(1,1,2,2)$-packing colorable, confirming that shorter $S$-sequences suffice in the subcubic case. Their work also established that every connected cubic graph admits a $(1,1,2,2,k)$-packing coloring in which at most one vertex receives color~$k$, for any $k \ge 3$, thereby improving the result of~\cite{11}. This line of research continues and unifies earlier contributions by Mortada and Togni, who introduced and systematically developed the study of $S$-packing colorings of $k$-saturated subcubic graphs~\cite{12,13,14,15}, as well as by Mortada and El~Zein~\cite{AM3}, whose recent work further refined the saturation framework. The present paper builds upon these foundations, combining the saturation-based perspective with the newly introduced local girth parameter to obtain a unified theory of $S$-packing colorability.

Subsequently, Bre\v{s}ar, Kuenzel, and Rall~\cite{5} showed that every claw-free cubic graph is $(1,1,2,2)$-packing colorable. 
Their result, obtained through a detailed structural decomposition, revealed the importance of local restrictions in achieving compact $S$-packings and inspired the present investigation. 

The notion of \emph{local girth} provides a natural parameter for describing the local cyclic structure of a vertex. 
For a vertex $v$, its local girth, denoted $g(v)$, is the length of the shortest cycle containing $v$, and it is supposed to be $\infty$ if $v$ belongs to no cycle. 
A claw-free cubic graph is precisely a cubic graph in which every vertex has local girth~3. 
Consequently, Bre\v{s}ar, Kuenzel, and Rall~\cite{5} admits a natural reformulation in terms of \emph{local girth}: 
Every cubic graph in which  every vertex has local girth~$3$ is $(1,1,2,2)$-packing colorable.

Motivated by this characterization, El Zein and Mortada~\cite{AM2} proved that every subcubic graph such that each 3-vertex has local girth at most 4 is $(1,2,2,2,2,2)$-packing colorable, extending the claw-free case beyond the cubic setting.

In this paper, we unify and extend these two directions---the study based on the saturation level and that guided by the local girth---by analyzing the \emph{impact of local girth on the $S$-packing coloring of $k$-saturated subcubic graphs}.
We develop a general framework in which the local girth  of $3$-vertices governs the admissible packing configurations within each saturation class. For a subcubic graph $G$, we denote by $g_3(G) = \max\{\mathrm{g}(v) : d_G(v)=3\}$ the local girth of 3-vertices, that is, the smallest length of a cycle containing any 3-vertex. Our results cover and generalize several known cases in a unified manner.

For clarity and completeness, all the obtained bounds for the $S$-packing colorings of $k$-saturated subcubic graphs under various local girth constraints are summarized in Table~\ref{table}. 
This table highlights the interplay between the two structural parameters---the saturation level and the local girth of $3$-vertices---and illustrates how their combination determines the admissible $S$-packing sequences. 
Some of these results hold with exceptions for a small number of specific graphs, which are later identified and treated separately in the corresponding sections. 
The table also presents the open problems that naturally arise from these results, showing the frontier between proven cases and conjectured extensions. 
In each case, the presented bounds are proven to be sharp, with explicit constructions provided whenever equality holds.

The proofs rely on a local and constructive analysis guided by the interaction between the saturation level and the local girth of $3$-vertices. 
For each considered class, we identify the structural constraints that determine the admissible $S$-packing patterns and verify that these constraints guarantee the existence of the corresponding colorings. The constructive nature of the most of these arguments naturally leads to explicit procedures that can be formulated as efficient algorithms for producing the required $S$-packing colorings.

\begin{table}[H]
\centering
\caption{Summary of $S$-packing colorability results for $k$-saturated subcubic graphs under local girth constraints. 
Each entry indicates the best known $S$-sequence guaranteeing an $S$-packing coloring}
\label{table}
\begin{tabular}{!{\vrule width 1pt}c!{\vrule width 1pt}c!{\vrule width 1pt}c!{\vrule width 1pt}c!{\vrule width 1pt}c!{\vrule width 1pt}}
\Xhline{1pt}
\textbf{\makecell{Subcubic Graph\\ Class}} & $g_{3}(\cdot)$ & \textbf{Proven} & \textbf{Disproven} & \textbf{Conjectured} \\
\Xhline{1pt}
\multirow{4}{*}{$0$-Saturated} 
& \multirow{4}{*}{$4$} & $(1,2,2,3)$ & $(1,2,2)$ & $(1,2,2,4)$ \\
\cline{3-5}
&  & \makecell{$(2,2,2,2,3)$ \\ except $\mathcal{G}_1$} & $(2,2,2,2)$ & \makecell{$(2,2,2,2,4)$ \\ except $\mathcal{G}_1$}\\
\cline{3-5}
&   & $(1,2,3,4,5,k);\; k\geq 6$ &  & $(1,2,3,4)$ \\
\Xhline{1pt}
\multirow{6}{*}{$1$-Saturated} 
& \multirow{4}{*}{$3$} & $(1,1,3,k);\; k \geq 3$ & $(1,1,3)$ & $(1,2,3,3)$ \\
\cline{3-5}
&   & $(1,2,2,3)$ & $(1,2,2)$ & $(1,2,2,4)$ \\
\cline{3-5}
&   & $(2,2,2,2,3)$ & $(2,2,2,2)$ & $(2,2,2,2,4)$ \\
\cline{3-5}
&   & $(1,2,3,4,5,k);\; k\geq 6$ &  & $(1,2,3,4,5)$ \\
\cline{2-5}
& $4$ & $(1,2,2,2)$ & $(1,2,2)$ & \\
\cline{2-5}
& $\geq 5$ &  & $(1,2,2,2)$ & \\
\Xhline{1pt}
\multirow{5}{*}{$2$-Saturated} 
& \multirow{3}{*}{$3$} & $(1,1,2)$ & $(1,1,3,3)$ & \\
\cline{3-5}
&   & $(1,2,2,2)$ & $(1,2,2,3)$ &  \\
\cline{3-5}
&   & $(2,2,2,2,2)$ & $(2,2,2,2)$ & $(2,2,2,2,3)$ \\
\cline{2-5}
& \multirow{2}{*}{$\geq 4$} &  & $(1,2,2,2)$ & $(1,2,2,2,2)$ \\
\cline{3-5}
&   &  & $(2,2,2,2,2)$ & $(1,1,2)$ \\
\Xhline{1pt}
\multirow{3}{*}{$(3,0)$-Saturated} & 
\multirow{3}{*}{$3$} & $(1,1,2,4)$ & $(1,1,3,3)$ & $(1,1,2)$ \\
\cline{3-5}
& & $(1,2,2,2,2)$ & & \\
\cline{3-5}
& & $(2,2,2,2,2,2)$ & & \\
\Xhline{1pt}
$(3,1)$-Saturated & $3$ & $(1,1,2,3)$ & $(1,1,3,3)$ & \\
\Xhline{1pt}
$(3,2)$-Saturated & $3$ & $(1,1,3,3,3)$ &  $(1,1,3,3)$ & \\
\Xhline{1pt}
\multirow{5}{*}{$(3,3)$-Saturated} & 
\multirow{3}{*}{$3$} & $(1,1,2,2)$ ~\cite{5}&  \makecell{$(1,1,3,3,3)$ \\for $\mathcal{G}_{11}$} & \makecell{$(1,1,3,3,3)$ \\except $\mathcal{G}_{11}$}\\
\cline{3-5}
& & & & $(1,1,2,3)$ \\
\cline{3-5}
& & & & $(2,2,2,2,2,2,2)$\\
\cline{2-5}
& \multirow{2}{*}{$4$} & $(1,2,2,2,2,2)$ ~\cite{AM2}&  & $(1,1,2,3)$ \cite{AM2}\\
\cline{3-5}
& & & & $(1,2,2,2,2)$ \cite{AM2} \\
\Xhline{1pt}
\end{tabular}
\end{table}

The paper is organized according to the saturation level of the considered subcubic graphs. Section~2 deals with graphs of low saturation (0- and 1-saturated), where the influence of short local cycles on admissible $S$-packings is first analyzed. Section~3 focuses on 2-saturated graphs, establishing refined bounds that depend on the interaction between 3-vertices and local girth constraints. Section~4 is devoted to the most intricate case of 3-saturated graphs, where we classify several subclasses and determine optimal $S$-packing colorings. Finally, Section~5 provides a synthesis of all obtained results, discusses sharpness and exceptional graphs, and concludes with open problems motivated by the observed connection between local girth and packing colorability.

\section{$0$-Saturated and $1$-Saturated}
We start by a lemma on graphs having connected components paths and/or cycles.
\begin{lemma}\label{lemma 1}
    Let $G$ be a connected graph such that $\Delta(G)\leq 2$. Then,\begin{itemize}
        \item[(i)] $G$ is $(1,1,k)$-packing colorable for every $k\geq 1$.
        \item[(ii)] $G$ is $(1,2,2)$-packing colorable unless $G=C_5$,
        \item[(iii)] $G$ is $(2,2,2,2)$-packing colorable unless $G=C_5$.
        \item[(iv)] $G$ is $(1,2,4,5,k)$-packing colorable for every $k\geq 6$.
    \end{itemize}
\end{lemma}
\begin{proof}
    It is sufficient to suppose that $G$ is a cycle of order $n$. Trivially, (i) holds. Here is the proof of (ii).\\
    \textbf{Case 1:} $n=12m+r$ for some $m\geq 0$ and $r\in \{0,3,6,9\}$.\\
    Here $n$ is a multiple of $3$. Color the vertices of $G$ using the sequence $12_a2_b\dots 12_a2_b.$\\
    \textbf{Case 2:} $n=12m+r$ for some $m\geq 0$ and $r\in \{1,4,7,10\}$.\\
    Here $n=3p+1$ for some $p\geq 1$. Color the first four vertices by $12_a12_b$, then the number of the remaining vertices is a multiple of $3$, color them using the sequence $12_a2_b\dots 12_a2_b$.\\
    \textbf{Case 3:} $n=12m+2$ for some $m\geq 1$.\\
    Here $n=4p+2$ for some $p\geq 3$. Color the first six vertices using the sequence $2_a12_b2_a12_b$, then the number of the remaining vertices is a multiple of $4$, color them using the sequence $12_a12_b\dots 12_a12_b$.\\
    \textbf{Case 4:} $n=12m+8$ for some $m\geq 0$.\\
    Here $n=4p$ for some $p\geq 2$. Use the sequence $12_a12_b\dots 12_a12_b$.\\
    \textbf{Case 5:} $n=12m+11$ for some $m\geq 0$.\\
    Here $n=4p+3$ for some $p\geq 2$. Color the first three vertices using the sequence $12_a2_b$, then the number of the remaining vertices is a multiple of $4$, color them using the sequence $12_a12_b\dots 12_a12_b$.\\
    \textbf{Case 6:} $n=12m+5$ for some $m\geq 1$.\\
    Color the first seventeen vertices using the sequence $2_b12_a12_b12_a2_b12_a2_b12_a2_b12_a1$, then the number of the remaining vertices is a multiple of $12$, color them using the sequence $2_b12_a12_b12_a12_b12_a1\dots 2_b12_a12_b12_a12_b12_a1$.\\
    \\
    The proof of (ii) can be extended to a proof of (iii) by recoloring the vertices that are colored by $1$ using the colors $2_c$ and $2_d$ properly. Here is the proof of (iv).\\
    \textbf{Case 1:} $n=4m$ for some $m\geq 1$.\\
    If $m$ is even, color the vertices of $G$ using the sequence $1,2,1,4,1,2,1,5,\dots$. If $m$ is odd, color the vertices of $G$ using the sequence $1,2,1,4,1,2,1,5,\dots ,1,2,1,5,1,2,1,k$.\\
    \textbf{Case 2:} $n=4m+1$ for some $m\geq 1$.\\
    If $m$ is even, use the sequence $1,2,1,4,1,2,1,5,\dots ,1,2,1,4,1,5,2,1,k$. If $m$ is odd, use the sequence $1,2,1,4,1,2,1,5,\dots ,1,2,1,5,1,4,2,1,k$.\\
    \textbf{Case 3:} $n=4m+2$ for some $m\geq 1$.\\
    If $m$ is even, use the sequence $1,2,1,4,1,2,1,5,\dots ,1,2,1,4,1,5,1,2,1,k$. If $m$ is odd, use the sequence $1,2,1,4,1,2,1,5,\dots ,1,2,1,5,1,4,1,2,1,5$.\\
    \textbf{Case 4:} $n=4m+3$ for some $m\geq 1$.\\
    If $m$ is even, use the sequence $1,2,1,4,1,2,1,5,\dots ,1,2,1,4,1,5,1,2,4,1,k$. If $m$ is odd, use the sequence $1,2,1,4,1,2,1,5,\dots ,1,2,1,5,1,4,1,2,k,1,5$.\\
\end{proof}

The following theorem establishes several results concerning $0$- and $1$-saturated subcubic graphs in which  the local girth of 3-vertices, $g_3(\cdot)$, is bounded by $4$ and $3$, respectively.
\begin{theorem}\label{theorem 1}
Let $G$ be a subcubic graph that satisfies one of the following:\begin{itemize}
    \item[(a)] $G$ is $0$-saturated and $g_3(G)\leq 4$,
    \item[(b)] $G$ is $1$-saturated and $g_3(G)=3$.
\end{itemize}
Then, $G$ is $S$-packing colorable, when one of the following holds: \begin{itemize}
    \item[(i)] $S=(1,1,3,k)$ for any $k\geq 3$,
    \item[(ii)] $S=(1,2,2,3)$,
    \item[(iii)] $S=(2,2,2,2,3)$ and $G\neq \mathcal{G}_1$,
    \item[(iv)] $S=(1,2,3,4,5,k)$ for any $k\geq 6$.
\end{itemize}
\end{theorem}
\begin{proof}
Without loss of generality, we may suppose that $G$ is connected. An $S$-packing coloring of $G$ is said to be good if the color $3$ is used only for $2$-vertices that are adjacent to a $3$-vertex and that are in a cycle of order at most $g_3(G)$. We will prove, by induction on the number of $3$-vertices in $G$, that $G$ has a good $S$-packing coloring, except for the following cases, referred to as \emph{bad exceptions}:\begin{itemize}
    \item[(1)] $G=C_5$ and $S=(2,2,2,2,3)$,
    \item[(2)] $G=C_5$ and $S=(1,2,2,3)$,
    \item[(3)] $G=\mathcal{G}_1$ and $S=(1,2,2,3)$,
    \item[(4)] $G=\mathcal{G}_2$ and $S=(1,2,2,3)$,
    \item[(5)] $G=\mathcal{G}_3$ and $S=(1,2,2,3)$.
\end{itemize}
First, suppose that $G$ has no $3$-vertices. Then, the result follows from Lemma \ref{lemma 1}. Now, assume that $G$ contains $3$-vertices. Let $x$ be a $3$-vertex and $x_1,x_2,x_3$ be the neighbors of $x$ such that $x_1$ and $x_2$ belong to a smallest cycle containing $x$, called $C_x$. As $G$ is $i$-saturated, where $i\in \{0,1\}$, we may assume that $x_1$ is a $2$-vertex. Let $G'=G-x_1$. Clearly, $dist_{G'}(u,v)=dist_G(u,v)$ for all $u,v\in V(G')$. Moreover, $G'$ is $i'$-saturated, where $i'\leq i$, while $G$ is $i$-saturated. In addition, $g_3(G')\leq g_3(G)$. Thus, by the induction hypothesis, $G'$ has a good $S$-packing coloring unless $G'$ and $S$ form a bad exception.\\
\\
\textbf{Case 1:} $G'$ has a good $S$-packing coloring.\\
Let $u_1$ be a vertex in $G'$ colored by $3$. Then, $u_1$ is a $2$-vertex adjacent to a $3$-vertex in $G'$, say $u$, such that $u_1$ and $u$ belong to a cycle in $G'$ of order at most $g_3(G)$, called $C_u$.

\begin{claim}
    $dist_G(x_1,u_1)\geq 2$.
\end{claim}
\begin{proof}
On the contrary, suppose that $dist_G(x_1,u_1)=1$, that is $x_1$ and $u_1$ are adjacent. Then, $u_1$ is a $3$-vertex in $G$. As $G$ is $i$-saturated, where $i\in \{0,1\}$, and $u_1$ is adjacent to a $3$-vertex other than $x$, then $u_1$ is not adjacent to $x$. If $u_1=x$, then $G$ is not $0$-saturated. Hence, $g_3(G)=3$ and $x_1$ is adjacent to $x_2$. Thus, $x_2$ is not a $3$-vertex in $G'$ and so $u=x_3$. Consequently, $x$ is adjacent to two $3$-vertices, $x_2$ and $x_3$, a contradiction. So, $u_1\neq x$. Recall that $x,x_1,x_2$ belong to the cycle $C_x$ of order at most $4$. Now, as $x_1$ and $u_1$ are adjacent and since $G$ is $i$-saturated, $u_1$ belongs to the same cycle, that is $C_x=C_u$. As $u_1$ is neither $x$ nor $x_2$, $C_x$ is of order $4$ and $u_1$ is adjacent to $x_2$. Thus, $g_3(G)=4$ and $G$ is $0$-saturated. But $u_1$ and $u$ are two adjacent $3$-vertices, a contradiction.
\end{proof}

\begin{claim}
    $dist_G(x_1,u_1)\geq 3$.
\end{claim}
\begin{proof}
On the contrary, suppose that $dist_G(x_1,u_1)=2$. Then, $x_1$ and $u_1$ have a common neighbor. If $u_1$ is adjacent to $x$ and $x_2$, that is $u=x_2$, then $G$ is not $0$-saturated and so $g_3(G)=3$. Hence, $V(C_u)=\{x,x_2,u_1\}$. Recall that $u$ is a $3$-vertex in $G'$, then the neighbor of $u$ other than $u_1$ and $x$ is not $x_1$. Thus, $C_x$ is of order at least $4$, a contradiction. Now, if $u_1$ is adjacent to $x$ but not to $x_2$, then $C_u$ is of order $4$ and so $G$ is $0$-saturated. Thus, $x_2$ is a $2$-vertex. As $u$ is a $3$-vertex in $G'$, $C_x$ is of order at least $5$, a contradiction. Similarly, if $x_2$ is a common neighbor of $x_1$ and $u_1$, we may find a contradiction. Finally, assume that $x_1$ and $u_1$ have a common neighbor $y$ other than $x$ and $x_2$. Then, $C_x=xx_1yx_2$ and so $g_3(G)=4$ and $G$ is $0$-saturated. Hence, $x_2$ is a $2$-vertex. Thus, $C_u$ is of order at least $5$, a contradiction.
\end{proof}

\begin{claim}
    $dist_G(x_1,u_1)\geq 4$.
\end{claim}
\begin{proof}
On the contrary, suppose that $dist_G(x_1,u_1)=3$. Then, there exist a neighbor $y$ of $x_1$ and a neighbor $z$ of $u_1$ such that $y$ and $z$ are adjacent. Note that $z$ is not adjacent to $x_1$. First, suppose that $y=x$. If $z=x_2$, then $z$ is adjacent to a third vertex that belongs to $C_x$, and so $G$ is $1$-saturated with $g_3(G)\geq 4$, a contradiction. So, we may assume that $z=x_3$. As $z\in V(C_u)$ and neither $x_1$ nor $x_2$ is adjacent to $u_1$, $z$ has a third neighbor, say $w$, other than $x,x_1,x_2$, such that $w\in V(C_u)$. Now, $x$ and $z$ are two adjacent $3$-vertices. Then, $G$ is $1$-saturated and $g_3(G)=3$. Hence, $x_2$ is a $2$-vertex and $G=\mathcal{G}_3$. In this case, one can check that $G$ has a good $S$-packing coloring unless $S=(1,2,2,3)$, which is a bad exception. Now, assume that $y\neq x$. If $y=x_2$, then $y$ is a $3$-vertex and so $z$ is a $2$-vertex. Thus, the order of $C_u$ is at least $5$, a contradiction. Hence, we may assume that $y\neq x_2$. Thus, $C_x=xx_1yx_2$ and $G$ is $0$-saturated. Again, $z$ is a $2$-vertex which leads to a contradiction.
\end{proof}\\
\\
Now, as $dist_G(x_1,u_1)\geq 4$, we can color $x_1$ by $3$ to obtain a good $S$-packing coloring of $G$.\\
\\
\textbf{Case 2:} $G'$ and $S$ form a bad exception.\\
If $G'=\mathcal{G}_i$ for some $i\in \{1,2,3\}$, then $G$ is not $1$-saturated, a contradiction. So, we may assume that $G'=C_5$. Then, $G=\mathcal{G}_2$ or $G=\mathcal{G}_1$. If $G=\mathcal{G}_2$ and $S=(1,2,2,3)$, then we are in a bad exception. If $G=\mathcal{G}_2$ and $S=(2,2,2,2,3)$, then $G$ has an $S$-packing coloring (see Figure \ref{figure 2}). If $G=\mathcal{G}_1$, then $S=(1,2,2,3)$, and so we are in a bad exception.\\
\\
We are left with the bad exceptions that are represented in Figure \ref{figure 2}.
\end{proof}

As corollaries, the studied graphs have a bounded packing chromatic number.
\begin{corollary}
    If $G$ is a $0$-saturated subcubic graph such that $g_3(G)\leq 4$, then $\chi_{\rho} (G)\leq 6$.
\end{corollary}
\begin{corollary}
    If $G$ is a $1$-saturated subcubic graph such that $g_3(G)=3$, then $\chi_{\rho} (G)\leq 6$.
\end{corollary}

Now, we study $1$-saturated subcubic graphs such that $g_3(\cdot)$ is bounded by $4$.
\begin{theorem}
    Every $1$-saturated subcubic graph such that $g_3(G)\leq 4$ is $(1,2,2,2)$-packing colorable.
\end{theorem}
\begin{proof}
A $(1,2,2,2)$-packing coloring using the colors $1,2_a,2_b,2_c$ is said to be good if the color $2_c$ is used only to color some $2$-vertices that belong to a cycle of order at most $4$. We will prove that every $1$-saturated subcubic graph such that $g_3(G)\leq 4$ has a good $(1,2,2,2)$-packing coloring. On the contrary, let $G$ be a counter-example of minimum order. Clearly, $G$ contains a $3$-vertex. Let $x$ be a $3$-vertex. First, suppose that $\text{g}(x)=3$. Let $y$ be a neighbor of $x$ such that $d(y)=2$ and $y$ belongs to a cycle of order $3$. Clearly, $G-y$ is a $1$-saturated subcubic graph such that $g_3(G-y)\leq 4$. Then, $G-y$ has a good $(1,2,2,2)$-packing coloring. Moreover, $y$ is at a distance of at least $3$ from any $2$-vertex that belongs to a cycle of order $3$. Thus, one can color $y$ by $2_c$ in order to obtain a good $(1,2,2,2)$-packing coloring of $G$, a contradiction. Now, assume that $\text{g}(x)=4$. Let $y_1,y_2,y_3$ be the vertices in the cycle of order $4$ that contains $x$ such that $y_1$ and $y_2$ are adjacent to $x$. If $d(y_1)=d(y_2)=2$, then in a similar way as above, one can extend a good $(1,2,2,2)$-packing coloring of $G-y_1$ to $G$ to obtain a contradiction. So, we may assume that $d(y_1)=3$ and $d(y_2)=2$. Clearly, $y_3$ is a $2$-vertex. Let $G'=G-\{y_2,y_3\}$. Again, $G'$ is $1$-saturated such that $g_3(G')\leq 4$. Then, $G'$ has a good $(1,2,2,2)$-packing coloring. Moreover, $y_2$ and $y_3$ are at a distance of at least $3$ from any $2$-vertex other than them that are in a cycle of order of at most $4$. Now, obviously $x$ or $y_1$ is not colored by $1$. Without loss of generality, suppose that $x$ is not colored by $1$. Color $y_2$ by $1$ and $y_3$ by $2_c$ in order to obtain a good $(1,2,2,2)$-packing coloring of $G$, a contradiction.
\end{proof}

\section{$2$-Saturated}
In this section, we prove three results. We start by a lemma about the structure of a $2$-saturated subcubic graph that will be a minimum counter example for the three results if they are supposed to be false. Then, we introduce two easy lemmas that will be useful in the coloring of the graphs in the main results.

\begin{lemma}\label{lemma 2-sat}
 Let $G$ be a connected $2$-saturated subcubic graph such that $g_3(G)=3$, $\delta(G)=2$, $\Delta(G)=3$, $G$ contains no diamond, and for every cycle $C$ in $G$ we have $V(G)\neq N[C]$. Then, $G$ contains a $2$-packing $Y$ such that $G-Y$ is a set of disjoint paths $P_1,\dots ,P_{\alpha}$ of order at least $4$, unless $P_i$ contains a $2$-vertex that belongs to a triangle, satisfying the following for all $x\in P_i$ and $y\in P_j$, $i\neq j$:\begin{itemize}
     \item[(i)] $dist_G(x,y)\geq 3$ whenever $x$ and $y$ are interior vertices,
     \item[(ii)] $dist_G(x,y)\geq 3$ whenever $x$ and $y$ are ends.
 \end{itemize}
\end{lemma}
\begin{proof}
We will study the structure of $G$.

\begin{claim}\label{claim 2-sat 2}
Let $C$ be a cycle and $xy$ be a chord in $C$. Then, $dist_C(x,y)=2$.
\end{claim}
\begin{proof}
    As $xy$ is a chord in $C$, $dist_C(x,y)\geq 2$. Suppose to the contrary that $dist_C(x,y)\geq 3$. Let $x_1,x_2$ be the neighbors of $x$ in $C$. Hence, $x_1x_2\in E(G)$. Thus, the three neighbors of $x$ are $3$-vertices, a contradiction.
\end{proof}\\
\\
A cycle is said to be good if its order is at least $4$ and it contains no chords.
\begin{claim}\label{claim 2-sat 3}
Let $C$ be a good cycle and $x\in V(G)\setminus V(C)$ such that $N(x)\cap N(C)\neq \emptyset$. Then, $|N(x)\cap V(C)|=2$ and the two neighbors of $x$ in $C$ are adjacent in $C$.
\end{claim}
\begin{proof}
    Let $y\in N(x)\cap V(C)$ and let $y_1,y_2$ be the neighbors of $y$ in $C$. As $C$ is good, $y_1y_2\notin E(G)$. Then, either $x$ is adjacent to $y_1$ or $y_2$. As $G$ is $2$-saturated, $x$ cannot be adjacent to both $y_1$ and $y_2$. The result follows.
\end{proof}\\
\\
If $C$ is a good cycle and $x\in V(G)\setminus V(C)$ such that $N(x)\cap V(C)\neq \emptyset$, then $x$ is said to be $C$-attached. 

\begin{claim}\label{claim 2-sat 4}
Let $C$ be a good cycle and $x,y$ be $C$-attached vertices. Denote by $x_1,x_2$ (resp., $y_1,y_2$) the neighbors of $x$ (resp., $y$) that belong to $C$ such that $x_1,x_2,y_1,y_2$ are in this order in $C$ (i.e., $x_1x_2v_1\dots v_ky_1y_2$ is a path in $C$, for some vertices $v_1,\dots ,v_k$). Then,\begin{itemize}
    \item[(i)] If $x$ is a $3$-vertex, then $dist_G(x_1,y_1)\geq 3$.
    \item[(ii)] If $x$ and $y$ are $2$-vertices, then $dist_G(x,y)\geq 3$.
\end{itemize}
\end{claim}
\begin{proof}
    Note that $x_1,x_2,y_1,y_2$ are $3$-vertices. Let $x'_1$ and $x'_2$ be the third neighbors of $x_1$ and $x_2$, respectively. First, suppose that $x$ is a $3$-vertex. Here, $x'_1$ and $x'_2$ are $2$-vertices. Then, $x'_1$ and $x'_2$ are distinct from $y_1$ and $y_2$. Moreover, $x$ is not adjacent to $y_1$. Thus, $x_1$ and $y_1$ are neither adjacent nor have a common neighbor. Therefore, $dist_G(x_1,y_1)\geq 3$. Now, assume that $x$ and $y$ are $2$-vertices. Clearly, $x_1$ and $x_2$ are different from $y_1$ and $y_2$. Thus, $x$ and $y$ are neither adjacent nor have a common neighbor. Therefore, $dist_G(x,y)\geq 3$.
\end{proof}

\begin{claim}\label{claim 2-sat 5}
Let $C$ be a good cycle and $x$ be a $C$-attached $3$-vertex. Denote by $x_1,x_2,x_3$ the neighbors of $x$ such that $x_1,x_2\in V(C)$. Let $a,b,c$ be three vertices that form a triangle in $G-V(C)$. Then, $dist_G(x_1,a)\geq 3$.
\end{claim}
\begin{proof}
    Recall that $x_3$ is a $2$-vertex. Then, $x_3\neq a$. Similarly, the neighbor of $x_1$ in $C$ other than $x_2$ is a $2$-vertex. Thus, $x_1$ and $a$ have no common neighbors. The result follows.
\end{proof}

\begin{claim}
If $G$ contains a cycle of order at least $4$, then $G$ contains a good cycle.
\end{claim}
\begin{proof}
    Let $C$ be a cycle of order at least $4$. If $C$ has no chords, then we are done. Otherwise, let $xy$ be a chord in $C$. By Claim \ref{claim 2-sat 2}, $dist_C(x,y)=2$. As $G$ contains no diamond, the order of $C$ is at least $5$. Assume that $C=v_1v_2\dots v_nv_1$ such that $v_1v_3$ is a chord. Let $v_{i_1}v_{i_1+2},\dots , v_{i_r}v_{i_r+2}$ be the other chords in $C$ such that $i_1<\dots <i_r$, if they exist. Then, the cycle $v_1v_3\dots v_{i_1}v_{i_1+2}\dots v_{i_r}v_{i_r+2}\dots v_1$ is a good cycle in $G$.
\end{proof}\\
\\
Let $C_1$ be a good cycle in $G$ and $G_1=G-V(C_1)$. Recursively, for $i\geq 1$, define $C_i$ as a good cycle in $G_{i-1}$ and $G_i$ to be $G_{i-1}-V(C_i)$. This definition ends at some $p\geq 2$, where $G_p$ does not contain a good cycle. Note that if $G$ contains no good cycle, then $G$ contains no cycles of order at least $4$. In this case, we may assume that $G_p=G$.

If there exists $i\in \{1,\dots ,p\}$ such that every $C_i$-attached vertex is a $2$-vertex, then $V(G)=N[C_i]$, a contradiction. So, we may assume that, for every $i\in \{1,\dots ,p\}$, there exists a $C_i$-attached $3$-vertex. Let $i\in \{1,\dots ,p\}$. Let $A_i$ be the set of $C_i$-attached $2$-vertices. Let $x^i_1,\dots ,x^i_{k_i}$ be the  $C_i$-attached $3$-vertices. Denote by $y_j,z_j$ the neighbors of $x^i_j$ that belong to $C_i$, for every $j\in \{1,\dots ,k_i\}$, and suppose that $y_1,z_1,y_2,z_2,\dots ,y_{k_i},z_{k_i}$ appear in this order on $C_i$. Let $B_i=\{y_1,\dots ,y_{k_i}\}$ and $X=\bigcup_{1\leq i \leq p} (A_i\cup B_i))$. By Claim \ref{claim 2-sat 4} and Claim \ref{claim 2-sat 5}, $X$ is a $2$-packing. Note that if $G$ contains no good cycles, $X$ is an empty set that is supposed a $2$-packing also.

As $G_p$ contains no good cycle, every cycle in $G_p$ is a triangle. Let $D_1$ be the set formed by exactly one $2$-vertex from every triangle in $G_p$ that contains a $2$-vertex. Clearly, $X\cup D_1$ is a $2$-packing. Now, let $G_p^1$ be the graph obtained by removing the vertices of every triangle containing a $2$-vertex from $G_p$. Note that every cycle in $G_p^1$ is a triangle formed by three $3$-vertices. Moreover, every neighbor of a vertex in a cycle in $G_p^1$, outside the cycle, is a $2$-vertex. Now, we will pick a $3$-vertex from every cycle in $G_p^1$ in order to form a set $D_2$. The set $Y:=X\cup D_1\cup D_2$ will be a $2$-packing that we are aiming for.

\begin{claim}
The set $D_2$ exists such that $Y$ is a $2$-packing.
\end{claim}
\begin{proof}
    A $t$-vertex is a $3$-vertex that lies in a triangle in $G_p^1$. Let $x$ be a $t$-vertex. Consider the BFS tree $T$ of $x$. Let $x_1,x_2,x_3$ be the neighbors of $x$ such that $x_1$ and $x_2$ are adjacent. Let $D_2=\{x\}$ for the moment, then we will add vertices to $D_2$. A $t$-vertex is said to be good if its father in $T$ is a $2$-vertex. Otherwise, it is said to be bad. Note that every triangle in $G_p^1$ is formed by a good $t$-vertex and two bad $t$-vertices. Add to $D_2$ every good $t$-vertex that belongs to the branch of $T$ of root $x_1$ (resp., $x_2$). Note that the distance between any two vertices in $D_2$ is at least $3$ in $T$. Moreover, as $G_p^1$ contains no cycle of order at least $4$, the distance between every vertex in $D_2$ is also at least $3$ in $G$. Now, add to $D_2$ a bad $t$-vertex among every pair of adjacent bad $t$-vertices in the branch of $T$ of root $x_3$. It is easy to check that $D_2$ is a $2$-packing. Moreover, every vertex in $D_2$ is at a distance of at least $3$ from every vertex in $X\cup D_1$. Note that if $T$ does not contain every $t$-vertex in $G_p^1$, we can consider several BFS trees in the same way.
\end{proof}\\
\\
By the definition of $Y$, $G-Y$ is a set of disjoint paths $P_1,\dots ,P_{\alpha}$.

\begin{claim}
    The path $P_i$ is of order at least $4$ for every $i\in \{1,\dots , \alpha\}$ unless $P_i$ contains a $2$-vertex that belongs to a triangle..
\end{claim}
\begin{proof}
Let $i\in \{1,\dots ,\alpha\}$. There exists $x\in P_i$ such that $x$ has a neighbor $y\in Y$. Indeed, otherwise, $\Delta(G)\leq 2$. Let $z$ be the other neighbor of $y$ that belongs to $P_i$. If $y$ is a $3$-vertex, then $x$ (resp., $z$) is adjacent to a $2$-vertex $x_i$ (resp., $z_i$) that belongs to $P_i$. Then, $P_i$ is of order at least $4$. Otherwise, $y$ is a $2$-vertex. Without loss of generality, we may suppose that $x$ is a $3$-vertex. Let $x_1$ be the neighbor of $x$ other than $y$ and $z$. If $x_1$ is a $3$-vertex, then $x_1$ is contained in a triangle that one of its vertices is a vertex in $Y$ and the last vertex belongs to $P_i$. Thus, $P_i$ is of order at least $4$. If $x_1$ is a $2$-vertex, then the neighbor of $x_1$ other than $x$ belongs to $P_i$. Hence, $P_i$ is of order at least $4$ unless $z$ is a $2$-vertex. Finally, if $x_1$ is a $1$-vertex, then, as $V(G)\neq N[xyz]$, $z$ is a $3$-vertex and has a neighbor that belongs to $P_i$. The result follows.
\end{proof}

Now, if $x\in Y$ is a $3$-vertex, then $x$ is contained in a triangle of $3$-vertices. Moreover, if $x\in Y$ is a $2$-vertex, then $x$ is contained in a triangle. Now, let $x\in P_i$ and $y\in P_j$, $i\neq j$, be two interior vertices. Suppose to the contrary that $x$ and $y$ have a common neighbor $z$. Clearly, $z\in Y$. If $z$ is a $2$-vertex, then $x$ and $y$ are adjacent, a contradiction. So, $z$ is a $3$-vertex. Thus, $z$ is contained in a triangle of $3$-vertices. Without loss of generality, suppose that $y$ is contained in that triangle and let $u$ be the third vertex of the triangle. As $u\neq x$ and $G$ is $2$-saturated, $x$ is a $2$-vertex. Then, $x$ is an end of $P_i$, a contradiction. Therefore, $dist_G(x,y)\geq 3$.

Finally, let $x\in P_i$ and $y\in P_j$, $i\neq j$, be two ends. Suppose to the contrary that $x$ and $y$ have a common neighbor $z$. Again, $z\in Y$. If $z$ is a $2$-vertex, then $x$ and $y$ are adjacent, a contradiction. So, $z$ is a $3$-vertex. Thus, $z$ is contained in a triangle of $3$-vertices. Without loss of generality, suppose that $y$ is contained in that triangle and let $u,v$ be other neighbors of $y$. As $Y$ is a $2$-packing, $u$ and $v$ are not in $Y$. Then, $u,v\in P_j$ and $y$ is not an end of $P_j$, a contradiction. Therefore, $dist_G(x,y)\geq 3$.
\end{proof}\\
\\
The following two lemmas are trivial.

\begin{lemma}\label{lemma paths (1,2,2)}
Let $P=v_1\dots v_n$ be a path. If $n\geq 3$, then $P$ has a $(1,2,2)$-packing coloring such that $v_1$ and $v_n$ are colored by $1$.
\end{lemma}
\begin{lemma}\label{lemma paths (2,2,2,2)}
Let $P=v_1\dots v_n$ be a path. If $n\geq 4$, then $P$ has a $(2,2,2,2)$-packing coloring such that $v_1$ and $v_n$ are colored by the same color and neither of $v_2,\dots ,v_{n-1}$ is colored by that color.
\end{lemma}

\begin{theorem}
Every $2$-saturated subcubic graph such that $g_3(G)=3$ is $(1,1,2)$-packing colorable.
\end{theorem}
\begin{proof}
On the contrary, let $G$ be a counter-example of minimum order. Obviously, $G$ is connected and contains a $3$-vertex.

\begin{claim}\label{claim 2-sat 1}
We have $\delta(G)=2$.
\end{claim}
\begin{proof}
    Suppose to the contrary that there exists a vertex $u$ in $G$ such that $d(u)= 1$. The graph $G-u$ is $2$-saturated subcubic such that $g_3(G-u)=3$. Then, $G-u$ is $(1,1,2)$-packing colorable. Hence, either $1_a$ or $1_b$ is not assigned to the neighbor of $u$. By coloring $u$ using the appropriate color $1_a$ or $1_b$ we obtain a $(1,1,2)$-packing coloring of $G$, a contradiction.
\end{proof}

\begin{claim}
$G$ contains no diamond.
\end{claim}
\begin{proof}
    Suppose to the contrary that $G$ contains a diamond and let $x_1,x_2,x_3,x_4$ be the vertices of a diamond such that $x_1x_3$ is a chord. As $G$ is connected and $x_1$ and $x_3$ are $3$-vertices, either $x_2$ or $x_4$ has a neighbor outside the diamond. Without loss of generality, suppose that $x_2$ is adjacent to a vertex $y$ outside the diamond. As $G$ is $2$-saturated, $x_4$ and $y$ are $2$-vertices. consider a $(1,1,2)$-packing coloring of $G-\{x_1,x_2,x_3,x_4,y\}$. Clearly, $y$ has one colored neighbor. Then, either $1_a$ or $1_b$ is not assigned to its colored neighbor, say $1_a$. Color $y,x_1,x_2,x_3,x_4$ by $1_a,1_a,1_b,2,1_b$, respectively. We obtain a $(1,1,2)$-packing coloring of $G$, a contradiction.
\end{proof}

\begin{claim}
If $C$ is a cycle, then $V(G)\neq N[C]$.
\end{claim}
\begin{proof}
    Suppose to the contrary that there exists a cycle $C$ such that $V(G)=N[C]$. If $C$ is an even cycle, then we can color the vertices of $C$ by $1_a$ and $1_b$. The vertices in $N(C)$ are at a distance of at least $3$ from each other, then we can color them by $2$ to obtain a $(1,1,2)$-packing coloring, a contradiction. Now, assume that $C$ is an odd cycle. Then, $C$ contains a $2$-vertex. Let $C=v_1v_2\dots v_k$ such that $v_1$ and $v_2$ are $3$-vertices and $v_3$ is a $2$-vertex. Let $x_1,\dots ,x_m$ be the vertices in $N(C)$ such that $x_1$ is adjacent to $v_1$ and $v_2$. Clearly, $v_2$ is at a distance of at least $3$ from any vertex among $x_2,\dots ,x_m$. Color the vertices $v_2,x_2,\dots ,x_m$ by $2$. The remaining vertices induce a path. Then, we can color the remaining vertices by $1_a$ and $1_b$. We obtain a $(1,1,2)$-packing coloring of $G$, a contradiction.
\end{proof}\\
\\
Now, by Lemma \ref{lemma 2-sat}, $G$ contains a $2$-packing $Y$ such that $G-Y$ is a set of disjoint paths $P_1,\dots ,P_{\alpha}$. Color the vertices of $P_1,\dots ,P_{\alpha}$ by $1_a$ and $1_b$ to obtain a $(1,1,2)$-packing coloring of $G$, a contradiction.
\end{proof}

\begin{theorem}\label{2-sat (1,2,2,2)}
Every $2$-saturated subcubic graph such that $g_3(G)=3$ is $(1,2,2,2)$-packing colorable.
\end{theorem}
\begin{proof}
On the contrary, let $G$ be a counter-example of minimum order. Obviously, $G$ is connected and contains a $3$-vertex.

\begin{claim}
$\delta(G)\ge 2$.
\end{claim}
\begin{proof}
Suppose, to the contrary, that there exists a vertex $x$ of degree one. Let $y$ be its (unique) neighbor. By minimality of $G$, the graph $G-x$ admits a $(1,2^3)$-packing coloring with palette $\{1,2_a,2_b,2_c\}$. If $y$ is not colored $1$, then color $x$ by $1$ and so the coloring is extended to $G$, a contradiction. Otherwise, $y$ has color $1$; pick $i\in\{a,b,c\}$ such that no vertex at distance~2 from $x$ has color $2_i$ (one exists because at most two vertices are at distance~2 from $x$). Color $x$ by $2_i$, which preserves all packing constraints. This extends the coloring to $G$, a contradiction.
\end{proof}

\begin{claim}
$G$ contains no diamond.
\end{claim}
\begin{proof}
    Suppose to the contrary that $G$ contains a diamond and let $x_1,x_2,x_3,x_4$ be the vertices of a diamond such that $x_1x_3$ is a chord. As $G$ is connected and $x_1$ and $x_3$ are $3$-vertices, either $x_2$ or $x_4$ has a neighbor outside the diamond. Without loss of generality, suppose that $x_2$ is adjacent to a vertex $y$ outside the diamond. As $G$ is $2$-saturated, $x_4$ and $y$ are $2$-vertices. consider a $(1,1,2)$-packing coloring of $G-\{x_1,x_2,x_3,x_4,y\}$. Clearly, $y$ has one colored neighbor, say $z$. If $z$ is colored by $1$, then at least one of the colored neighbors of $z$ are not colored by a color $2$, say $2_a$. Now, color $y,x_1,x_2,x_3,x_4$ by $2_a,2_b,1,2_c,1$, respectively. We obtain a $(1,2,2,2)$-packing coloring of $G$, a contradiction. Otherwise, $z$ is colored by a color $2$, say $2_a$. Color $y,x_1,x_2,x_3,x_4$ by $1,1,2_b,2_c,2_a$, respectively. Again, $G$ is $(1,2,2,2)$-packing colorable, a contradiction.
\end{proof}

\begin{claim}
If $C$ is a cycle, then $V(G)\neq N[C]$.
\end{claim}
\begin{proof}
    Suppose to the contrary that there exists a cycle $C$ such that $V(G)=N[C]$. If $C$ is of order $5$, it can be easily checked that $G$ is $(1,2,2,2)$-packing coloring, which is a contradiction. Otherwise, $C$ is not of order $5$. Then, by Lemma \ref{lemma 1}, $C$ has a $(1,2,2)$-packing coloring using the colors $1,2_a,2_b$. By coloring the vertices of $N(C)$ by $2_c$ we obtain a $(1,2,2,2)$-packing coloring of $G$, a contradiction.
\end{proof}\\
\\
Now, by Lemma \ref{lemma 2-sat}, $G$ contains a $2$-packing $Y$ such that $G-Y$ is a set of disjoint paths $P_1,\dots ,P_{\alpha}$ of order at least $4$ such that $dist_G(x,y)\geq 3$ for all interior vertices $x\in P_i$ and $y\in P_j$, $i\neq j$. By Lemma \ref{lemma paths (1,2,2)}, one can color the vertices of the paths $P_1,\dots ,P_{\alpha}$ by $1,2_a,2_b$ such that the ends of them are colored by $1$. We obtain a $(1,2,2,2)$-packing coloring of $G$, a contradiction.
\end{proof}

\begin{theorem}\label{2-sat (2,2,2,2,2)}
Every $2$-saturated subcubic graph such that $g_3(G)=3$ is $(2,2,2,2,2)$-packing colorable.
\end{theorem}
\begin{proof}
On the contrary, let $G$ be a counter-example of minimum order. Obviously, $G$ is connected and contains a $3$-vertex.

\begin{claim}
$\delta(G)\geq 2$.
\end{claim}
\begin{proof}
    Suppose to the contrary that there exists a vertex $u$ in $G$ such that $d(u)= 1$. The graph $G-u$ is $2$-saturated subcubic such that $g_3(G-u)=3$. Then, $G-u$ is $(2,2,2,2,2)$-packing colorable. As the number of vertices that are at a distance of at most $2$ from $u$ is at most $3$, one can find a color $2_i$ for $u$ that extends the $(2,2,2,2,2)$-packing coloring to $G$, a contradiction.
\end{proof}

\begin{claim}
$G$ contains no diamond.
\end{claim}
\begin{proof}
    Suppose to the contrary that $G$ contains a diamond and let $x_1,x_2,x_3,x_4$ be the vertices of a diamond such that $x_1x_3$ is a chord. As $G$ is connected and $x_1$ and $x_3$ are $3$-vertices, either $x_2$ or $x_4$ has a neighbor outside the diamond. Without loss of generality, suppose that $x_2$ is adjacent to a vertex $y$ outside the diamond. As $G$ is $2$-saturated, $x_4$ and $y$ are $2$-vertices. consider a $(2,2,2,2,2)$-packing coloring of $G-\{x_1,x_2,x_3,x_4\}$. Without loss of generality, we may suppose that $y$ and its colored neighbor are colored by $2_a$ and $2_b$, respectively. Color $x_1,x_2,x_3,x_4$ by $2_c,2_d,2_e,2_b$, respectively. We obtain a $(2,2,2,2,2)$-packing coloring of $G$, a contradiction.
\end{proof}

\begin{claim}
If $C$ is a cycle, then $V(G)\neq N[C]$.
\end{claim}
\begin{proof}
    Suppose to the contrary that there exists a cycle $C$ such that $V(G)=N[C]$. If $C$ is of order $5$, it can be easily checked that $G$ is $(2,2,2,2,2)$-packing coloring, which is a contradiction. Otherwise, $C$ is not of order $5$. Then, by Lemma \ref{lemma 1}, $C$ has a $(2,2,2,2)$-packing coloring using the colors $2_a,2_b,2_c,2_d$. By coloring the vertices of $N(C)$ by $2_e$ we obtain a $(2,2,2,2,2)$-packing coloring of $G$, a contradiction.
\end{proof}\\
\\
Now, by Lemma \ref{lemma 2-sat}, $G$ contains a $2$-packing $Y$ such that $G-Y$ is a set of disjoint paths $P_1,\dots ,P_{\alpha}$ of order at least $4$, unless $P_i$ contains a $2$-vertex that belongs to a triangle, satisfying the following for all $x\in P_i$ and $y\in P_j$, $i\neq j$:\begin{itemize}
     \item[(i)] $dist_G(x,y)\geq 3$ whenever $x$ and $y$ are interior vertices,
     \item[(ii)] $dist_G(x,y)\geq 3$ whenever $x$ and $y$ are ends.
 \end{itemize}
Suppose that $P_1,\dots ,P_{\beta}$ are of order at least $4$ and $P_{\beta+1},\dots ,P_{\alpha}$ are of order $3$. By Lemma \ref{lemma paths (2,2,2,2)}, one can color the vertices of the paths $P_1,\dots ,P_{\beta}$ by $2_a,2_b,2_c,2_d$ such that only the ends of them are colored by $2_a$. In addition, color the vertices of $P_{\beta+1},\dots ,P_{\alpha}$ by $2_a,2_b,2_c$ such that the end tat does not belong to a triangle is colored by $2_a$. Color the vertices of $Y$ by $2_e$ to obtain a $(2,2,2,2,2)$-packing coloring of $G$, a contradiction.
\end{proof}

\section{$3$-Saturated}

As a Corollary of Theorem \ref{2-sat (1,2,2,2)} and Theorem \ref{2-sat (2,2,2,2,2)}, we can deduce the following.
\begin{corollary}
    Every $(3,0)$-saturated subcubic graph $G$ such that $g_3(G)=3$ is $(1,2,2,2,2)$-packing colorable and $(2,2,2,2,2,2)$-packing colorable.
\end{corollary}
\begin{proof}
Let $G$ be a $(3,0)$-saturated subcubic graph such that $g_3(G)=3$ and let $x^1,\dots, x^s$ be the heavy vertices of $G$. Let $i\in \{1,\dots ,s\}$. We denote bye $x^i_1,x^i_2,x^i_3$ the neighbors of $x^i$ such that $x^i_1$ and $x^i_2$ are adjacent. We denote by $y^i$ and $z^i$ the neighbors of $x^i_3$ such that $y^i$ is a $2$-vertex. Obviously, $y^i$ and $z^i$ are adjacent. If is clear that, for every distinct integers $i,j\in \{1,\dots, s\}$, we have $dist_G(y^i,y^j)\geq 3$. Moreover, the graph $G'$ obtained by deleting the vertices $y^1,\dots ,y^s$ from $G$ is $2$-saturated such that $g_3(G')=3$. By Theorem \ref{2-sat (1,2,2,2)} (resp., Theorem \ref{2-sat (2,2,2,2,2)}), $G'$ is $(1,2,2,2)$-packing colorable (resp., $(2,2,2,2,2)$-packing colorable). Color the vertices of $G'$ by $1,2_b,2_c,2_d$ (resp., $2_b,2_c,2_d,2_e,2_f$) and the vertices $y^1,\dots ,y^s$ by $2_a$ in order to obtain a $(1,2,2,2,2)$-packing coloring (resp., $(2,2,2,2,2,2)$-packing coloring) of $G$.
\end{proof}

Below, we study the $(1,1,2,k)$-packing colorings for $(3,0)$- and $(3,1)$-saturated subcubic graphs.
\begin{theorem}
    Let $G$ be a $(3,i)$-saturated subcubic graph, $0\leq i\leq 1$, such that $g_3(G)=3$, then:\begin{enumerate}
        \item  $G$ is $(1,1,2,4)$-packing colorable if $i=0$.
        \item $G$ is $(1,1,2,3)$-packing colorable if $i=1$.
    \end{enumerate}
\end{theorem}
\begin{proof}
Let $G$ be a $(3,i)$-saturated subcubic graph, $0\leq i\leq 1$, such that $g_3(G)=3$. For $i=0$, we will prove  $G$ has a $(1,1,2,4)$-packing coloring, and for $i=1$ we will prove $G$ has a $(1,1,2,3)$-packing coloring. Without loss of generality, suppose that $G$ is a connected graph.

Our plan is to prove the existence of a 2-packing $X$  and a 4-packing (resp., 3-packing) $Y$  in $G$  if $i=0$ (resp., $i=1$) such that the subgraph induced by $V(G)\setminus (X\cup Y)$ is a bipartite one.

A 3-vertex $x$ is said to be rich if $x$ is contained in two triangles. Let $x$ be a rich vertex and let $x_1$, $x_2$ and $x_3$ be the neighbors of $x$. As $x$ is rich, then one of its neighbors is also a rich vertex; assume this neighbor is $x_1$. Since $G$ is $(3,i)$-saturated with $i\in \{0,1\}$, neither $x_2$ nor  $x_3$  can be  rich. In fact, if $x_2$ is a rich vertex, then $x_2x_3\in E(G)$ and so $x$ and its neighbors are all heavy vertices, a contradiction. Moreover, if $G$ is $(3,1)$-saturated and if $x_2$ and $x_3$ are 3-vertices, then both $x_2$ and $x_3$ are non-heavy vertices.

For each vertex $x$ in $G$, we define the weight of $x$, $w(x)$, such that \\
$w(x)=\begin{cases} k_1 & \text{if $x$ is a 2-vertex or  a rich 3-vertex}\\
k_2 & \text {if $x$ is a non heavy 3-vertex and not rich}\\
k_3 & \text{if $x$ is a  heavy 3-vertex and  not rich.}
\end{cases}$

with $k_1$, $k_2$, $k_3$ are real numbers satisfying $k_1>k_2>k_3$ and $k_1>2k_3$.

Let $X,\; Y\subseteq V(G)$ such that $X$ and $Y$ are disjoint. We say that $(X,Y)$ is a packing pair of $G$ if: (1) $X$ is a 2-packing, (2) $Y$ is a 4-packing (resp., 3-packing) if $i=0$ (resp., $i=1$), (3) each triangle in $G$ has at most one vertex in $X\cup Y$, and (4) each vertex in $X\cup Y$ is a vertex of a triangle. For a packing pair $(X,Y)$, we define  $\theta(X, Y)$ to be the sum of the weights of the vertices in $X\cup Y$, and $\gamma(X, Y)$ to be the number of triangles in the subgraph $G[V(G)\setminus (X\cup Y)]$. The packing pair $(X,Y)$ is said to be a maximum packing pair if $\theta(X, Y)$ is maximum.
Let $(X,Y)$ be a maximum packing pair such that $\gamma(X, Y)$ is minimum. Let $G'=G[V(G)\setminus (X\cup Y)]$. We have the following results:
\begin{claim}\label{wt}
    If $x$ is a 3-vertex in $X\cup Y$ such that $x$ is not a rich one, then each other vertex of the triangle containing $x$ is a 3-vertex.
\end{claim}
\begin{proof}
 Suppose to the contrary that a vertex $y$ of the triangle containing $x$ is a 2-vertex. Then if $x\in X$ (resp., $x\in Y$), we get $(X',Y)$  (resp., $(X,Y')$) is a packing pair with $\theta(X',Y)\geq \theta(X,Y)-k_2+k_1>\theta (X,Y)$ (resp., $\theta(X,Y')\geq\theta(X,Y)-k_2+k_1>\theta (X,Y)$), where $X'=(X\setminus \{x\})\cup \{y\}$ (resp., $Y'=(Y\setminus\{x\})\cup \{y\})$, a contradiction.
\end{proof}

\medskip

We are going to prove that  $\gamma(X,Y)=0$.  
 Suppose to the contrary that $\gamma(X, Y)>0$ and let $T=abc$ be a triangle in $G'$. Then, by the maximality of $\theta(X, Y)$, for each vertex $u$ in $T$ there exists a vertex $x$ in $X$ (resp., $y$ in $Y$) such that $dist_G(u,x)<3$ (resp., $dist_G(u,y)<5$ if $i=0$ and $dist_G(u,y)<4$ if $i=1$).
 We have the following observation:
\begin{claim}
    No vertex in $T$ is adjacent to a vertex in $X$.
\end{claim}
\begin{proof}
    Suppose to the contrary that there exists a vertex $x$ in $X$ adjacent to a vertex in $T$; assume this vertex is  $a$.  By the definition of the packing pair, $x$ is a vertex of a triangle, say $T'=xyz$. Assume first that $V(T)\cap V(T')=\emptyset$. Thus, $x$ is not a rich vertex, and so, by Claim \ref{wt}, $y$ and $z$ are both 3-vertices. Consequently, $x$ is a heavy vertex. We will study two cases according to the saturation of $G$:
    \begin{enumerate}
        \item $i=0$.\\
        In this case, $a$ is a non heavy vertex since $G$ is $(3,0)$-saturated. Without loss of generality, suppose $c$ is a 2-vertex. Then $b$ is also adjacent to a vertex in $X$, since otherwise $(X',Y)$ is a packing pair with $\theta(X',Y)= \theta(X,Y)-k_3+k_1>\theta (X,Y)$, where $X'=(X\setminus \{x\})\cup \{c\}$, a contradiction. Let $x'$ be the neighbor of $b$ in $X$. Clearly, $x'$ is a vertex of a triangle not intersecting $T$, and so $x'$ is not a rich vertex. Thus,  by Claim \ref{wt}, $x'$ is a heavy vertex. Consequently, $(X',Y)$ is a packing pair with $\theta(X',Y)= \theta(X,Y)-2k_3+k_1>\theta (X,Y)$, where $X'=(X\setminus \{x,x'\})\cup \{c\}$, a contradiction. 
    \item $i=1$.\\ We need here to distinguish two cases:
    \begin{itemize}
        \item $b$ and $c$ are 3-vertices.\\
        In this cases, $a$ is a heavy vertex, and so, since $x$ is a heavy vertex, we get that $b$, $c$, $y$ and $z$  are all non heavy vertices. Thus, the neighbor of $b$ (resp., $c$) other than $a$ is a 2-vertex. We will study first the case when  $b$ and $c$ have a common neighbor other than $a$, say $x'$. In this case, $b$ (resp., $c$) is at distance at least four from each vertex in $Y\setminus\{x'\}$. Thus, if $x'\notin Y$, we get $(X,Y')$ is a packing pair, where $Y'=Y\cup\{b\}$, a contradiction. Thus $x'\in Y$, and so $(X,Y')$ is a packing pair with $\theta(X,Y)=\theta(X,Y')$ but $\gamma(X,Y')<\gamma(X,Y)$, where $Y'=(Y\setminus\{x'\})\cup\{b\}$, a contradiction. Thus, $b$ and $c$ have no common neighbor other than $a$. Therefore,   the neighbor of $b$ (resp., of $c$) other than $a$ is a not a vertex of $X\cup Y$ since it is not contained in a triangle. Hence, $a$ is at distance at least three from each vertex in $X\setminus\{x\}$.  Let $y'$ (resp., $z'$) be the neighbor of $y$ (resp., $z$ ) other than $x$. Clearly, $y'$ and $z'$ are 2-vertices. Moreover, $y'$ has a neighbor in $X$, since otherwise  $(X', Y)$ is a packing pair with $\theta(X',Y)>\theta (X,Y)$, where $X'=(X\setminus \{x\})\cup \{y\}$, a contradiction. Similarly, $z'$ has a neighbor in $X$. Let $y''$ (resp., $z''$) be the neighbor of $y'$ (resp., $z'$) in $X$. By the definition of $X$, $y''$ (resp., $z''$) is a vertex of a triangle. Consequently, since   each triangle in $G$ has at most one vertex in $X\cup Y$, and since the neighbor of $b$ (resp., $c$) is not a vertex of a triangle, we deduce that $x$ is at distance at least four from each vertex of $Y$. Let $X'=(X\setminus \{x\})\cup \{a\}$ and $Y'=Y\cup \{x\}$, then $(X',Y')$ is a packing pair with $\theta(X',Y')>\theta (X,Y)$, a contradiction.
        \item $b$ and $c$ are not both 3-vertices.
        \\
        Without loss of generality, suppose that $c$ is a 2-vertex. Then we proceed as in case $i=0$ in order to reach a contradiction. 
    \end{itemize}
    \end{enumerate}
    Assume now that $V(T)\cap V(T')\neq \emptyset$. Clearly, $\{b,c\}\nsubseteq N(x)$, since otherwise each vertex of $T$ is heavy vertex, a contradiction. Thus, without loss of generality, suppose $T'=xab$. Then, $x$ and $c$ are not rich.  If $a$ is at distance at least three from each vertex in $X\setminus \{x\}$, then $(X',Y)$ is a packing pair with  $\theta (X',Y)\geq\theta(X,Y)$ but $\gamma (X',Y)<\gamma(X,Y)$, where $X'=(X\setminus\{x\})\cup \{a\}$, a contradiction. Hence, $c$ has a neighbor in $X$, say $x'$.  Clearly, the triangle containing $x'$ does not contain $c$. Thus, $c$ is a heavy vertex and $x'$ is  a heavy vertex by Claim \ref{wt}. Therefore, in case $G$ is $(3,0)$-saturated, we reached a contradiction.     Suppose now  $G$ is $(3,1)$-saturated. Note that $x$ is a 2-vertex, since otherwise each vertex of $T$ is a heavy one, a contradiction.  Consequently, $a$ is at distance at least four from each vertex in $Y$. Thus, we get $(X,Y')$ is a packing pair with $\theta(X,Y')>\theta(X,Y)$, where $Y'=Y\cup\{a\}$, a contradiction.\end{proof}

     Consequently, since for each vertex $u\in T$ there exists a vertex $x\in X$ with $dist_G(x,u)<3$, we get that each vertex of $T$ is a 3-vertex. Besides, $a'$ (resp., $b'$ and $c'$) has a neighbor in $X$, where $a'$ (resp., $b'$ and $c'$) is the neighbor of $a$ (resp., $b$ and $c$) outside $T$. Since $G$ is $(3,i)$-saturated, $0\leq i\leq 1$, then at least one vertex of $a'$, $b'$, and $c'$ is a 2-vertex; assume $a'$ is a 2-vertex. Since each vertex in $X$ is a vertex of a triangle, and since at most one vertex of a triangle is contained in $X\cup Y$, then:\begin{itemize}
     \item If $i=0$.\\
     Let $x$ be the neighbor of $a'$ in $X$, and let $T=xyz$ be the triangle containing $x$. Note that neither $y$ nor $z$ has a neighbor in $Y$. In fact, if $z$ has a neighbor in $Y$, say $z'$, then, by Claim \ref{wt} and by the definition of the packing pair, we get $z'$ and $z$ are both heavy vertices, a contradiction.   Thus $a$ is at distance at least five from each vertex in $Y$ and so $(X,Y\cup\{a\})$ is a packing pair, a contradiction.
         \item If $i=1$.\\
         We have $a$  is at distance at least four from each vertex in $Y$.  Therefore, $(X,Y\cup \{a\})$ is a packing pair, a contradiction. 
     \end{itemize}

    Thus, $\gamma(X,Y)=0$.

    Let $A^T, \; B^T\subseteq V(G)$, we say that $(A^T,B^T)$ is an extension packing pair of a packing pair $(A,B)$ if: (1) $A\subseteq A^T$ and $B\subseteq B^T$, (2) $A^T$ is a 2-packing and $B^T$ is a 4-packing if $i=0$ and 3-packing if $i=1$, (3) each vertex in $A^T\setminus A$ (resp., $B^T\setminus B)$ is a vertex of an odd cycle of length greater than 3 in  $G\setminus (A\cup B)$. Let $\phi (A^T,B^T)$  be the number of odd cycles in $G[V(G)\setminus (A^T\cup B^T])$. Without loss of generality,  suppose that the maximum packing pair $(X,Y)$ has an extension packing pair $(X^T,Y^T)$ such that $\phi(X^T,Y^T)$ is minimum among all $\phi(A^T,B^T)$, whenever $(A^T,B^T)$ is an extension packing pair of a maximum packing pair $(A,B)$  with $\gamma (A,B)=0$. Let $G''=G[V(G)\setminus (X^T\cup Y^T)]$. 

    We are going to prove that $\phi(X^T,Y^T)=0$. Suppose to the contrary that $\phi(X^T,Y^T)>0$ and let $C$ be an odd cycle in  $G''$. Clearly, $C$ is of length greater than three, and, by the minimality of $\phi(X^T,Y^T)$, for each vertex $x$ on $C$ there exists a vertex $u\in X^T$ (resp., $v\in Y^T$) such that $dist_G(x,u)<3$ (resp., $dist_G(x,v)<5$ if $i=0$ and $dist_G(x,v)<4$ if $i=1$). Besides, since $g_3(G)=3$ and $G'$ contains no triangles, then each 3-vertex $x$ on $C$ has a neighbor in $X\cup Y$, say $u$, such that $u$ and $x$ are contained in  a triangle. Note that each 3-vertex on $C$ is not rich. In fact, suppose to the contrary that a vertex $x$ on $C$ is contained in two triangles and let $y$ and $z$ be the neighbors of $x$ on $C$ and $u$ be the neighbor of $x$ in $X\cup Y$. Clearly, the two triangles containing $x$ are then $uxy$ and $uxz$. Consequently, $u$ and $x$ are both heavy vertices and so we reached a contradiction in case $i=0$. For $i=1$, $y$ and $z$ are then non heavy vertices, and so the neighbor of $y$ (resp., $z$) on $C$ other than $x$ is a 2-vertex. Thus, if $u\in X^T$ (resp., $u\in Y^T$), we get $x$ is at distance at least four from each vertex in $Y^T$ (resp. $X^T$), a contradiction. Therefore, for each vertex $x$ on $C$, there exists a vertex $u \in X \cup Y$ such that 
$xu \in E(G)$, and a unique neighbor of $x$ on $C$ is adjacent to $u$.  Hence, since $C$ is an odd cycle, we get $C$ has at least one 2-vertex.  Let $x$ be a 3-vertex on $C$ and let $y$ and $z$ be the two neighbors of $x$ on $C$ such that $z$ is a 2-vertex. Then:  
    \begin{claim}\label{nb}
        \begin{enumerate}
          \item $y$ and $x$ have a common neighbor in $X$.
          \item The common neighbor of $x$ and $y$ in $X$ is either a 2-vertex or a non heavy vertex having no neighbor in $X^T\cup Y^T$.
           \item $z$ is at distance at least four from each vertex in $Y^T$.
        \end{enumerate}
        
    \end{claim}
    
    \begin{proof}
    
    \begin{enumerate}
        \item   Since $x$ is a 3-vertex then $x$ is contained in a triangle. Thus $x$ has a neighbor in $X^T\cup Y^T$, say $u$. Consequently, $u$ is a common neighbor of $x$ and $y$. Since $G'$ has no triangles, then $u\in X\cup Y$. Suppose to the contrary that $u\in Y$, then $u$ has a neighbor in $X^T$, since otherwise $x$ is at distance at least three from each vertex in $X^T$, a contradiction. Let $u'$ be the neighbor of $u$ in $X^T$. By the definition of the extension packing pair, we get  $u'$ is a 3-vertex, and so $u'$ is contained in a triangle. Thus, $u'\in X$. Clearly, $u'$ is not a rich vertex, and so, by Claim \ref{wt}, $u'$  is a heavy vertex. Consequently, $u$ is a heavy vertex. Thus, we reached a contradiction for $i=0$. For $i=1$, we have $x$ and $y$ are non heavy vertices. Besides, the neighbor of $z$ other than $x$ has a neighbor in $X^T$, since otherwise $z$ is at distance at least three from each vertex in $X^T$, a contradiction. Thus, $x$ is at distance at least four from each vertex in $Y\setminus\{u\}$, and so $(X,Y')$ is a packing pair with $\theta(X,Y')>\theta(X,Y)$, where $Y'=(Y\setminus\{u\})\cup \{x\}$, a contradiction. 

        \item Suppose to the contrary, that $u$ is a heavy vertex. As $z$ is a 2-vertex, then $x$ is at distance at least three from each vertex in $X\setminus \{u\}$. Hence, $(X',Y)$ is a packing pair with $\theta(X',Y)>\theta(X,Y)$, where $X'=(X\setminus\{u\})\cup \{x\}$, a contradiction. 
        Thus, the neighbor of $u$ other than $x$ and $y$ if exists is a 2-vertex. Besides, this neighbor if exists is not in $X^T\cup Y^T$ since it is not contained in a cycle in $G'$.

        \item  The result here follows directly from (1) and (2).
    \end{enumerate}
    \end{proof}

   Consequently, for $i=1$, we get that $(X^T, Y_1^T)$ is an extension packing pair of $(X,Y)$ with $\phi(X_1^t,Y^T)<\phi(X^T,Y^T)$, where $Y_1^T=Y^T\cup \{z\}$, a contradiction. We proceed with $i=0$. Set $C=v_1\dots v_{2k+1}$ for some $k\in \mathbb{N}$. If each vertex in $N(C)\cap (X^T\cup Y^T)$ is a 2-vertex, then $G=N[C]$ and so it can be easily shown that $G$ is $(1,1,2,4)$-packing colorable. Else, let $x$ be a vertex on $C$ such that the neighbor of $x$ in $X^T\cup Y^T$ is a 3-vertex and let $u$ be this neighbor. Let $y$ be the neighbor of $x$ on $C$ such that $uxy$ is a triangle. Since $G$ is $(3,0)$-saturated, then  either the neighbor of $x$ or of $y$ on $C$ is a 2-vertex.  Without loss of generality, suppose that the neighbor of $x$ on $C$ is a 2-vertex. By Claim \ref{nb}, each neighbor of a vertex on $C$ in $V(G)\setminus C$ is in $X^T\cup Y^T$, then each cycle passing through $u$ in $G\setminus (X_1^T\cup Y^T)$ should pass through $x$ and $y$, where $X_1^T=X^T\setminus \{u\}.$ Let $X'=(X\setminus \{u\})\cup \{x\}$, then $(X',Y)$ is a packing pair with $\theta(X',Y)\geq \theta (X,Y)$. Besides, $(X_1^T\cup \{x\}, Y^T)$ is an extension packing pair of $(X',Y)$ with $\phi(X_1^T\cup \{x\}, Y^T)= \phi(X^T,Y^T)-1$ since $C$ no more exists in $G\setminus (X_1^T\cup\{x\},Y^T)$, a contradiction.

Thus, $G''$ is bipartite. Color the vertices of $Y^T$ by 4 if $i=0$ and by 3 if $i=1$, that of $X^T$ by 2, and the remaining vertices by $1_a$ and $1_b$. The obtained coloring is a $(1,1,2,4)$-packing coloring if $i=0$ and $(1,1,2,3)$-packing coloring if $i=1$.
\end{proof}

Finally, we study the $(1,1,3,3,3)$-packing colorability of $(3,2)$-saturated subcubic graphs.
\begin{theorem}
   Every $(3,2)$-saturated subcubic graph such that $g_3(G)=3$ is $(1,1,3,3,3)$-packing colorable.
\end{theorem}
\begin{proof}
On the contrary, let $G$ be a counter-example of minimum order. Obviously, $G$ is connected.
\begin{claim}\label{claim (3,2)-sat 1}
We have $\delta(G)=2$.
\end{claim}
\begin{proof}
    Suppose to the contrary that there exists a vertex $u$ in $G$ such that $d(u)= 1$. The graph $G-u$ is $(3,2)$-saturated subcubic such that $g_3(G-u)=3$. Then, $G-u$ is $(1,1,3,3,3)$-packing colorable. Hence, either $1_a$ or $1_b$ is not assigned to the neighbor of $u$. By coloring $u$ using the appropriate color $1_a$ or $1_b$ we obtain a $(1,1,3,3,3)$-packing coloring of $G$, a contradiction.
\end{proof}

\begin{claim}\label{claim (3,2)-sat 2}
    There are no adjacent two $2$-vertices in $G$.
\end{claim}
\begin{proof}
    Suppose to the contrary that there exist two $2$-vertices $u,v\in V(G)$ that are adjacent. Let $u_1$ (resp., $v_1$) be the other neighbor of $u$ (resp., $v$). Let $u_2$ and $u_3$ (resp., $v_2$ and $v_3$) be the other neighbors of $u_1$ (resp., $v_1$) when they exist. As $g_3(G)=3$, $u_1,u_2,u_3$ (resp., $v_1,v_2,v_3$) form a cycle when they exist. Let $u_4$ (resp., $u_5$) be the other neighbor of $u_2$ (resp., $u_3$). Consider the graph $G'$ obtained by deleting $u$ from $G$ and adding an edge between $u_1$ and $v$. Clearly, $G'$ is $(3,2)$-saturated subcubic such that $g_3(G')=3$. Consider a $(1,1,3,3,3)$-packing coloring of $G'$. If one of the colors $1_a$ or $1_b$ is not assigned to neither of the vertices $u_1$ nor $v$, by coloring $u$ using this color we obtain a $(1,1,3,3,3)$-packing coloring of $G$, which is a contradiction. So, we may assume that $u_1$ is colored by $1_a$ and $v$ is colored by $1_b$. If $1_b$ is not assigned to neither of the vertices $u_2$ and $u_3$, then we can recolor $u_1$ by $1_b$ and $u$ by $1_a$ to obtain a $(1,1,3,3,3)$-packing coloring of $G$, which is a contradiction. So, we may assume that $u_2$ is colored by $1_b$ and $u_3$ is colored by $3_a$. Similarly, we may assume that $v_1$, $v_2$, and $v_3$ are colored by $1_a$, $1_b$, and $3_b$, respectively. If $u_4$ is not colored by $1_a$, then we can recolor $u_2$ by $1_a$, $u_1$ by $1_b$, and $u$ by $1_a$ to obtain a contradiction. Hence, $u_4$ is colored by $1_a$. If $u_5$ is colored by $3_k$, for some $k\in \{b,c\}$, then we can recolor $u_3$ by $1_a$, $u_1$ by $3_a$, and $u$ by $1_a$ to obtain a contradiction. Thus, $u_5$ is colored by $1_k$, for some $k\in \{a,b\}$. Therefore, neither of the vertices that are at a distance of at most $3$ from $u$ is colored by $3_c$. By assigning $3_c$ to $u$ we obtain a $(1,1,3,3,3)$-packing coloring of $G$, a contradiction.
\end{proof}\\
\\
Let $T$ be a set of vertices that contains exactly one vertex from each triangle in $G$ and suppose that $T$ is chosen with minimum cardinality. That is, every vertex in $T$ belongs to a triangle in $G$ and neither of the other vertices of this triangle belongs to $T$. In addition, suppose that $T$ is chosen with the maximum number of $2$-vertices; among all choices, $T$ is chosen with the maximum number of non-heavy $3$-vertices. Let $R$ be the sungraph induced by $V(G)\setminus T$. Clearly, $\Delta(R)\leq 2$. Our goal is to prove that the vertices of $T$ can be colored using the colors $3_a,3_b,3_c$. Then, we will prove that in every odd cycle in $R$, there exists a vertex that is at a distance less than four from at most two vertices in $T$. Moreover, such vertices in every odd cycle are at a distance of at least four from each other. Then, we will color the vertices of $R$ using the colors $1_a,1_b,3_k$, where $3_k$ is used for one vertex from each odd cycle. Therefore, the $3_k$ color will be properly chosen from $3_a,3_b,3_c$ in order to obtain a $(1,1,3,3,3)$-packing coloring of $G$, which is a contradiction.\\
We will start by studying the relations between the vertices in $G^3[T]$.

\begin{claim}\label{claim (3,2)-sat 3}
    If $x\in T$ is a $2$-vertex, then $d_{G^3[T]}(x)\leq 2$.
\end{claim}
\begin{proof}
Let $x_1$ and $x_2$ be the neighbors of $x$. By Claim \ref{claim (3,2)-sat 2}, $x_1$ and $x_2$ are $3$-vertices. Let $y_1$ (resp., $y_2$) be the neighbor of $x_1$ (resp., $x_2$) other than $x$ and $x_2$ (resp., $x_1$). Let $i\in \{1,2\}$. If $y_i$ is a $3$-vertex, then $y_i$ belongs to a triangle, which at most one of its vertices belongs to $T$. Similarly, if $y_i$ is a $2$-vertex, then the neighbor of $y_i$ other than $x_i$ belongs to a triangle, which at most one of its vertices belongs to $T$. Hence, there is at most one vertex $z_i$ in $T$ at a distance of at most $3$ from $x$ such that the shortest $xz_i$-path passes through $y_i$. Thus, $d_{G^3[T]}(x)\leq 2$.
\end{proof}

\begin{claim}\label{claim (3,2)-sat 4}
    If $x\in T$ is a non-heavy $3$-vertex, then $d_{G^3[T]}(x)\leq 3$.
\end{claim}
\begin{proof}
Let $x_1,x_2,x_3$ be the neighbors of $x$ such that $x_1$ and $x_2$ are adjacent. As $T$ is chosen with the maximum number of $2$-vertices, $x_1$ and $x_2$ are $3$-vertices. As $x$ is not heavy, $x_3$ is a $2$-vertex. Let $i\in \{1,2,3\}$. As in Claim \ref{claim (3,2)-sat 3}, at most one vertex in $T$ such that the shortest path from $x$ to this vertex passes through $x_i$ is at a distance of at most $3$ from $x$. Thus, $d_{G^3[T]}(x)\leq 3$.
\end{proof}

\begin{remark}\label{remark (3,2)-sat}
    Let $x\in T$ be a non-heavy $3$-vertex such that $d_{G^3[T]}(x)=3$. Let $x_1,x_2,x_3$ be the neighbors of $x$ in $G$ and $y,z,w$ be the neighbors of $x$ in $G^3[T]$. Suppose that the shortest $xy$-path (resp., $xz$-path, $xw$-path) passes through $x_1$ (resp., $x_2$, $x_3$). Denote by $y_1$ (resp., $z_1,w_1$) the vertex in the triangle containing $y$ (resp., $z,w$) that is the closest, among the vertices of the triangle, to $x$. Note that $y_1$ (resp., $z_1$, $w_1$) may be $y$ (resp., $z$, $w$). Then, the shortest $x_1y_1$-path (resp., $x_2z_1$-path) is either an edge or a path of length two, where the second vertex is a $2$-vertex. Moreover, the shortest $x_3w_1$-path is an edge. 
\end{remark}

\begin{claim}\label{claim (3,2)-sat 5}
    If $x\in T$ is a heavy vertex, then $d_{G^3[T]}(x)\leq 4$. Moreover, three of the vertices that are adjacent to $x$ in $G^3[T]$ are $2$-vertices.
\end{claim}
\begin{proof}
Let $x_1,x_2,x_3$ be the neighbors of $x$ such that $x_1$ and $x_2$ are adjacent. As $T$ is chosen with the maximum number of $2$-vertices, $x_1$ and $x_2$ are $3$-vertices. Similarly, as $T$ is chosen with the maximum number of non-heavy $3$-vertices, $x_1$ and $x_2$ are heavy. Thus, $x_3$ is not heavy. Let $u$ and $v$ be the neighbors of $x_3$ other than $x$, such that $u$ is a $2$-vertex. Hence, $u\in T$ and $dist_G(x,u)=2$. Moreover, the third neighbor of $v$, which may be in $T$, is at a distance $3$ from $x$. On the other hand, let $y_1$ (resp., $y_2$) be the third neighbor of $x_1$ (resp., $x_2$). Let $i\in \{1,2\}$. Again, $y_i$ is not heavy, then the triangle containing $y_i$ contains a $2$-vertex called $z_i$. We have $dist_G(x,z_i)=3$. The result follows.
\end{proof}

\begin{claim}\label{claim (3,2)-sat 6}
    $G^3[T]$ contains a complete subgraph of order $4$ if and only if $G^3[T]=K_4$.
\end{claim}
\begin{proof}
Suppose that $G^3[T]$ contains a complete subgraph of order $4$ and let $x,y,z,w$ be its vertices. By Claims \ref{claim (3,2)-sat 3}, \ref{claim (3,2)-sat 4}, and \ref{claim (3,2)-sat 5}, $x,y,z,w$ are non-heavy $3$-vertices. Applying Remark \ref{remark (3,2)-sat} to $x,y,z,w$, we can deduce $T=\{x,y,z,w\}$. The result follows.
\end{proof}\\
\\
First, suppose that $G^3[T]=K_4$.
\begin{assertion}
    $G$ is $(1,1,3,3,3)$-packing colorable.
\end{assertion}
\begin{proof}
This case is represented in Figure \ref{figure 1}. Suppose that the red $x_1y_1$-path is of length $2$. Color $x$ and $y_2$ by $3_a$, $z_2$ by $3_b$, and $w_2$ by $3_c$. The remaining vertices form a bipartite graph, so we can color them by $1_a$ and $1_b$. Now, assume that the red $x_1y_1$-path and $x_2z_1$-path are of length one. As $G$ is $(3,2)$-saturated, either $y_2$ or $y_3$ is not heavy. Assume that $y_3$ is not heavy. Hence, the red $y_3z_2$-path is of length $2$. Color $x$ by $3_a$, $y_1$ and $z_3$ by $3_b$, and $w_1$ by $3_c$. Again, we can color the remaining vertices by $1_a$ and $1_b$. The last case is when $y_3$ is heavy. Hence, $z_2$ is heavy and $z_3$ is not heavy. This case is similar to the previous one.
\end{proof}
\begin{figure}[h!]
\centering
\begin{tikzpicture}[
    node/.style={circle, draw, fill=gray!10, minimum size=1mm},
    edge/.style={thick},
    green edge/.style={thick, draw=green},
    red edge/.style={thick, draw=red},
    blue edge/.style={thick, draw=blue},
    black edge/.style={thick, draw=black},
    dotted edge/.style={thick, dotted},
    weight/.style={font=\bfseries\small}
]

\node[node,label={above:$x$}] (x) at (0, 1) {};

\node[node,label={above:$x_1$}] (x1) at (-2, -1) {};
\node[node,label={above:$x_2$}] (x2) at (2, -1) {};
\node[node,label={above:$x_3$}] (x3) at (6, -1) {};

\node[node,color={red},fill=red!10] (u1) at (-2, -2.5) {};
\node[node,color={red},fill=red!10] (u2) at (2, -2.5) {};

\node[node,label={left:$y_1$}] (y1) at (-2, -4) {};
\node[node,label={left:$z_1$}] (z1) at (2, -4) {};
\node[node,label={left:$w_1$}] (w1) at (6, -4) {};

\node[node,label={above:$y_2$}] (y2) at (-3, -5.5) {};
\node[node,label={above:$y_3$}] (y3) at (-1, -5.5) {};
\node[node,label={above:$z_2$}] (z2) at (1, -5.5) {};
\node[node,label={above:$z_3$}] (z3) at (3, -5.5) {};
\node[node,label={above:$w_2$}] (w2) at (5, -5.5) {};
\node[node,label={above:$w_3$}] (w3) at (7, -5.5) {};

\node[node,color={red},fill=red!10] (u3) at (0, -5.5) {};
\node[node,color={red},fill=red!10] (u4) at (4, -5.5) {};
\node[node,color={red},fill=red!10] (u5) at (2, -7.2) {};

\draw (x) -- (x1);
\draw (x) -- (x2);
\draw (x) -- (x3);
\draw (x1) -- (x2);

\draw (y1) -- (y2);
\draw (y2) -- (y3);
\draw (y3) -- (y1);

\draw (z1) -- (z2);
\draw (z2) -- (z3);
\draw (z3) -- (z1);

\draw (w1) -- (w2);
\draw (w2) -- (w3);
\draw (w3) -- (w1);

\draw (x3) -- (w1);

\draw[red edge] (x1) -- (u1);
\draw[red edge] (u1) -- (y1);

\draw[red edge] (x2) -- (u2);
\draw[red edge] (u2) -- (z1);

\draw[red edge] (y3) -- (u3);
\draw[red edge] (u3) -- (z2);

\draw[red edge] (z3) -- (u4);
\draw[red edge] (u4) -- (w2);

\draw[red edge] (y2) to[out=-45, in=180] (u5);
\draw[red edge] (u5) to[out=0, in=-135] (w3);

\end{tikzpicture}

\caption{The red paths are either of length one or two. The vertex $y$ (resp., $z$, $w$) is one of the vertices $y_i$ (resp., $z_i$, $w_i$).}\label{figure 1}
\end{figure}
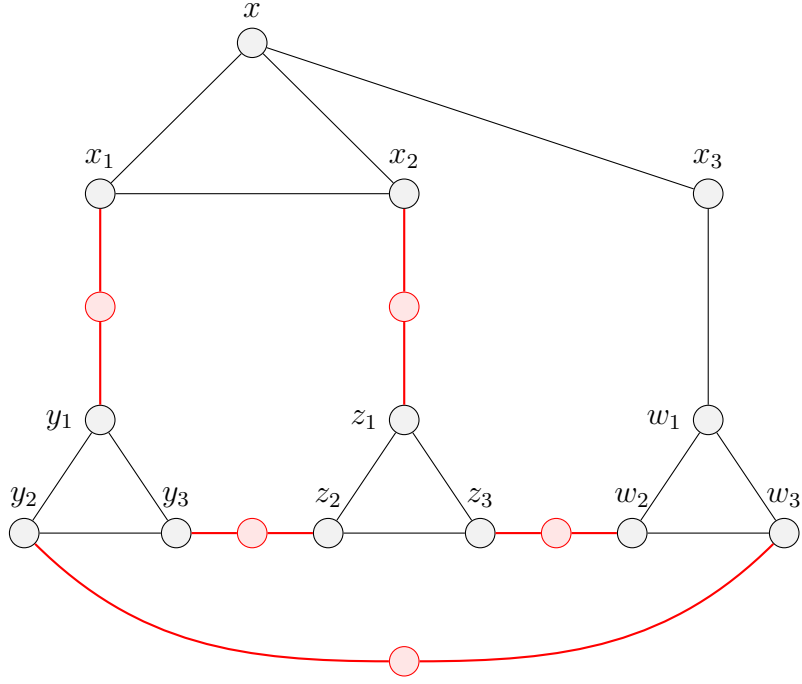\\
\\
So, we may assume that $G^3[T]$ does not contain a complete subgraph of order $4$.

\begin{claim}\label{claim (3,2)-sat 7}
    The graph $G^3[T]$ is $3$-colorable.
\end{claim}
\begin{proof}
Let $H$ be the set of heavy vertices in $T$ (they belong to $T$ but are heavy in $G$). By Claim \ref{claim (3,2)-sat 5}, every vertex in $H$ is adjacent in $G^3[T]$ to three vertices of degree $2$ in $G^3[T]$. Let $H'$ be the set of vertices of degree $2$ in $G^3[T]$ that are adjacent in $G^3[T]$ to a vertex in $H$. As $G^3[T]-H'$ has maximum degree $3$ and no complete subgraph of order $4$, by Brooks' Theorem, $G^3[T]-H'$ is $3$-colorable. A $3$-coloring of $G^3[T]-H'$ can be easily extended to $G^3[T]$ since a 2-vertex in  $G^3[T]$ has at most two neighbors in  $G^3[T]$.
\end{proof}\\
\\
Let $C$ be an odd cycle in $R$. Obviously, $C$ is of order of at least $5$. Moreover, $C$ contains a $2$-vertex. Let $X$ be a set formed by taking exactly one $2$-vertex from every odd cycle in $R$. 

\begin{claim}\label{claim (3,2)-sat 8}
    Let $x\in X$. At most two vertices in $T$ are at a distance less than $4$ from $x$.
\end{claim}
\begin{proof}
Let $u$ and $v$ be the neighbors of $x$. By Claim \ref{claim (3,2)-sat 2}, $u$ and $v$ are $3$-vertices. Let $u_1$ and $u_2$ (resp., $v_1$ and $v_2$) be the other neighbors of $u$ (resp., $v$). Then, $u_1u_2,v_1v_2\in E(G)$. Suppose that $u_1$ and $v_1$ belong to the odd cycle in $R$ that contains $x$. Then, $u_2,v_2\in T$. Hence, $u_2$ and $v_2$ are non-heavy vertices. Thus, $u_2$ (resp., $v_2$) is not adjacent to a vertex in $T$. Moreover, the third neighbor of $u_1$ (resp., $v_1$) belongs to the odd cycle in $R$ that contains $x$. That is, it does not belong to $T$. The result follows.
\end{proof}

\begin{claim}\label{claim (3,2)-sat 9}
    Let $x,y\in X$. Then, $dist_G(x,y)\geq 4$.
\end{claim}
\begin{proof}
Let $u,v,u_1,u_2,v_1,v_2$ be the vertices defined in Claim \ref{claim (3,2)-sat 8}. Clearly, $y\notin \{u,v,u_1,\\u_2,v_1,v_2\}$. Let $u_3$ (resp., $v_3$) be the third neighbor of $u_2$ (resp., $v_2$). Now, $u_3$ and $v_3$ are $2$-vertices. Then, $d_R(u_3)\leq 1$ and $d_R(v_3)\leq 1$. Thus, $y\notin \{u_3,v_3\}$. The result follows.
\end{proof}\\
\\
Now, color the vertices of $T$ by $3_a,3_b,3_c$. Note that, $R-X$ is bipartite. So, color the vertices of $R-X$ by $1_a$ and $1_b$. Finally, for each $x\in X$, there exists a color $3_k\in \{3_a,3_b,3_c\}$ that can be properly assigned to $x$. Then a $(1,1,3,3,3)$-packing coloring of $G$ is obtained, a contradiction.
\end{proof}

\section{Sharpness of the Results and Open Problems}

We proved that a $0$-saturated subcubic graph $G$ such that $g_3(G)\leq 4$ admits both a  $(1,2,2,3)$-packing coloring and a $(2,2,2,2,3)$-packing coloring, except $\mathcal{G}_1$. Both results are sharp in the sense that there are such graphs that are not $(1,2,2)$-packing colorable and others that are not $(2,2,2,2)$-packing colorable (Proposition \ref{prop 1}).

\begin{proposition}\label{prop 1}
    The graph $\mathcal{G}_1$ (resp., $\mathcal{G}_4$) is not $(1,2,2)$-packing colorable (resp., $(2,2,2,2)$-packing colorable).
\end{proposition}
\begin{proof}
The diameter of $\mathcal{G}_1$ is $2$, then at most one vertex can be colored by a color $2$. Moreover, in $\mathcal{G}_1$, there are no independent set of order $4$. Thus, at most five vertices can be colored using the colors $1,2_a,2_b$. Therefore, $\mathcal{G}_1$ is not $(1,2,2)$-packing colorable. Similarly, at most four vertices of $\mathcal{G}_4$ can be colored using four different colors $2$. Then, $\mathcal{G}_4$ is not $(2,2,2,2)$-packing colorable.
\end{proof}
\begin{figure}[h!]
\centering
\begin{tikzpicture}[
    node/.style={circle, draw, fill=gray!10, minimum size=1mm},
    edge/.style={thick},
    green edge/.style={thick, draw=green},
    red edge/.style={thick, draw=red},
    black edge/.style={thick, draw=black},
    dotted edge/.style={thick, dotted},
    weight/.style={font=\bfseries\small}
]

\node[node,label={left:$1$}] (1) at (1.5,0) {};
\node[node,label={right:$2_a$}] (2) at (2.5,0) {};
\node[node,label={right:$2_b$}] (3) at (3,1) {};
\node[node,label={above:$1$}] (4) at (2,1) {};
\node[node,label={left:$3$}] (5) at (1,1) {};
\node[node,label={above:$1$}] (6) at (2,2) {};

\draw (1) -- (2);
\draw (2) -- (3);
\draw (3) -- (4);
\draw (4) -- (5);
\draw (5) -- (1);
\draw (5) -- (6);
\draw (3) -- (6);

\node at (2, -1) {$\mathcal{G}_1$};


\node[node,label={left:$1,2_c$}] (x1) at (6,0.7) {};
\node[node,label={left:$2_a$}] (x2) at (6,1.7) {};
\node[node,label={above:$1,2_b$}] (x3) at (7,2.4) {};
\node[node,label={right:$3$}] (x4) at (8,1.7) {};
\node[node,label={right:$1,2_d$}] (x5) at (8,0.7) {};
\node[node,label={above:$2_b$}] (x6) at (7,0) {};

\draw (x1)--(x2);
\draw (x2)--(x3);
\draw (x2)--(x4);
\draw (x3)--(x4);
\draw (x4)--(x5);
\draw (x5)--(x6);
\draw (x6)--(x1);

\node at (7, -1) {$\mathcal{G}_2$};


\node[node,label={left:$2_a$}] (y1) at (11,0) {};
\node[node,label={left:$2_b$}] (y2) at (11,1) {};
\node[node,label={above:$1$}] (y3) at (12,0.5) {};
\node[node,label={above:$3$}] (y4) at (13,0.5) {};
\node[node,label={right:$2_a$}] (y5) at (14,0) {};
\node[node,label={right:$2_b$}] (y6) at (14,1) {};

\draw (y1) -- (y2);
\draw (y2) -- (y3);
\draw (y3) -- (y1);
\draw (y3) -- (y4);
\draw (y4) -- (y5);
\draw (y5) -- (y6);
\draw (y6) -- (y4);

\node at (12.5, -1) {$\mathcal{G}_3$};

\end{tikzpicture}
\caption{The graph $\mathcal{G}_1$: a $0$-saturated subcubic graph such that $g_3(\mathcal{G}_1)=4$ that is neither $(1,2,2)$-packing colorable nor $(2,2,2,2,2)$-packing colorable.\\
The graph $\mathcal{G}_2$: a $1$-saturated subcubic graph such that $g_3(\mathcal{G}_2)=3$ that is not $(1,2,2)$-packing colorable.\\
The graph $\mathcal{G}_3$: a $1$-saturated subcubic graph such that $g_3(\mathcal{G}_3)=3$.}\label{figure 2}
\end{figure}
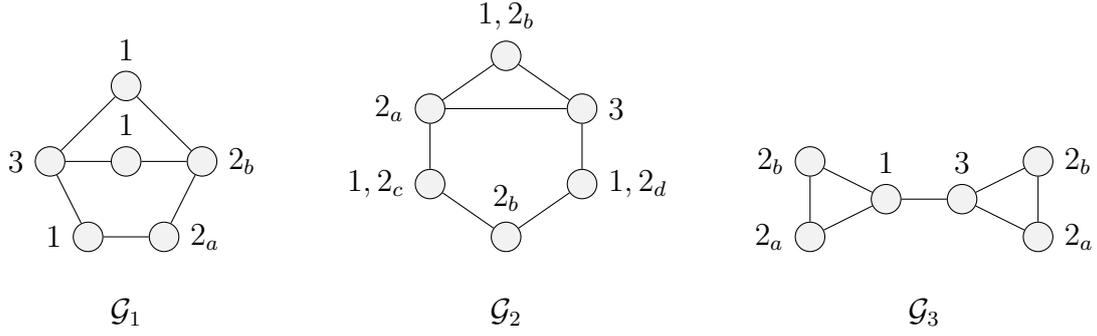

We proved that a $0$-saturated subcubic graph $G$ such that $g_3(G)\leq 4$ has $\chi_{\rho}(G)\leq 6$. We did not find such a graph $G$ such that $\chi_{\rho}(G)\geq 5$. Thus, we conjecture the following.

\begin{conjecture}
        If $G$ is a $0$-saturated subcubic graph such that $g_3(G)\leq 4$, then $\chi_{\rho} (G)\leq 4$.
\end{conjecture}

We proved that a $1$-saturated subcubic graph $G$ such that $g_3(G)=3$ is $(1,1,3,k)$-packing colorable for all $k\geq 3$. The result is sharp in the sense that there are such graphs that are not $(1,1,3)$-packing colorable (Proposition \ref{prop 2}) and there are such graphs that are not $(1,2,2)$-packing colorable (Proposition \ref{prop 3}).

\begin{proposition}\label{prop 2}
    The graph $\mathcal{G}_5$ is not $(1,1,3)$-packing colorable.
\end{proposition}
\begin{proof}
Suppose to the contrary that $\mathcal{G}_5$ is $(1,1,3)$-packing colorable. Each odd cycle must contain a vertex colored by $3$. Let $x,y,z$ be vertices in the three triangles that are colored by $3$. Clearly, neither of $x,y,z$ belongs to the cycle of order $9$. Thus, the cycle of order $9$ contains a fourth vertex colored by $3$, which is at a distance less than $3$ from one of the vertices $x,y,z$, a contradiction.
\end{proof}
\begin{figure}[h!]
\centering
\begin{tikzpicture}[
    node/.style={circle, draw, fill=gray!10, minimum size=1mm},
    edge/.style={thick},
    green edge/.style={thick, draw=green},
    red edge/.style={thick, draw=red},
    black edge/.style={thick, draw=black},
    dotted edge/.style={thick, dotted},
    weight/.style={font=\bfseries\small}
]

\node[node] (x1) at (-2,0) {};
\node[node] (x2) at (-1,1) {};
\node[node] (x3) at (-2,1) {};
\node[node] (x4) at (-3,1) {};
\node[node] (x5) at (-2,2) {};

\draw (x1) -- (x2);
\draw (x1) -- (x3);
\draw (x1) -- (x4);
\draw (x5) -- (x2);
\draw (x5) -- (x3);
\draw (x5) -- (x4);

\node at (-2, -1) {$\mathcal{G}_4$};


\node[node] (1) at (1,0) {};
\node[node] (2) at (2,0) {};
\node[node] (3) at (3,0) {};
\node[node] (4) at (4,0) {};
\node[node] (5) at (5,0) {};

\node[node] (6) at (4.5,1) {};
\node[node] (7) at (4,2) {};
\node[node] (8) at (3.5,3) {};
\node[node] (9) at (3,4) {};

\node[node] (10) at (2.5,3) {};
\node[node] (11) at (2,2) {};
\node[node] (12) at (1.5,1) {};

\draw (1) -- (2);
\draw (2) -- (3);
\draw (3) -- (4);
\draw (4) -- (5);
\draw (5) -- (6);
\draw (6) -- (7);
\draw (7) -- (8);
\draw (8) -- (9);
\draw (9) -- (10);
\draw (10) -- (11);
\draw (11) -- (12);
\draw (12) -- (1);

\draw (12) -- (2);
\draw (4) -- (6);
\draw (8) -- (10);

\node at (3, -1) {$\mathcal{G}_5$};


\node[node] (y1) at (7,0) {};
\node[node] (y2) at (8,0) {};
\node[node] (y3) at (9,0) {};
\node[node] (y4) at (10,0) {};
\node[node] (y5) at (11,0) {};

\node[node] (y6) at (10.35,1) {};
\node[node] (y7) at (9.7,2) {};
\node[node] (y8) at (9,3) {};

\node[node] (y9) at (8.3,2) {};
\node[node] (y10) at (7.65,1) {};

\draw (y1) -- (y2);
\draw (y2) -- (y3);
\draw (y3) -- (y4);
\draw (y4) -- (y5);
\draw (y5) -- (y6);
\draw (y6) -- (y7);
\draw (y7) -- (y8);
\draw (y8) -- (y9);
\draw (y9) -- (y10);
\draw (y10) -- (y1);

\draw (y10) -- (y2);
\draw (y4) -- (y6);

\node at (9, -1) {$\mathcal{G}_6$};

\end{tikzpicture}
\caption{The graph $\mathcal{G}_4$: a $0$-saturated subcubic graph such that $g_3(\mathcal{G}_4)=4$ that is not $(2,2,2,2)$-packing colorable.\\
The graph $\mathcal{G}_5$: a $1$-saturated subcubic graph such that $g_3(\mathcal{G}_5)=3$ that is not $(1,1,3)$-packing colorable.\\
The graph $\mathcal{G}_6$: a $1$-saturated subcubic graph such that $g_3(\mathcal{G}_6)=3$ that is not $(2,2,2,2)$-packing colorable.}\label{figure 3}
\end{figure}
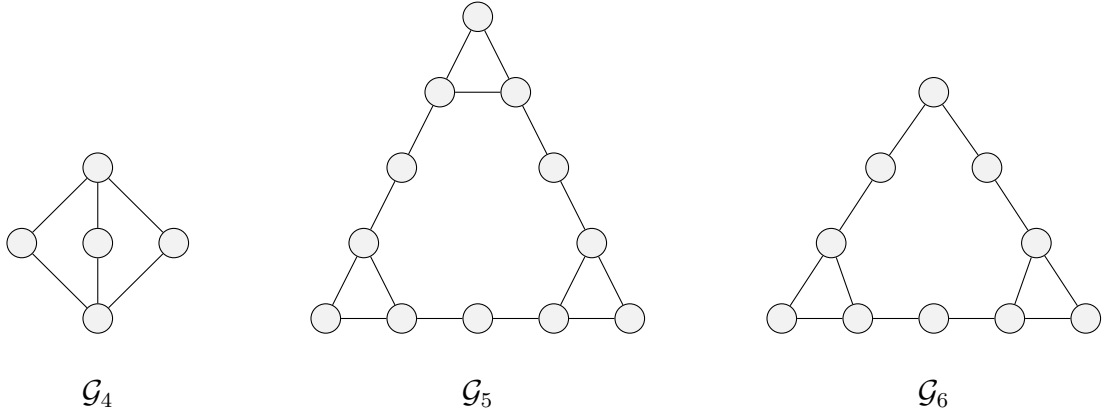

\begin{proposition}\label{prop 3}
   The graph $\mathcal{G}_2$ (resp., $\mathcal{G}_6$) is not $(1,2,2)$-packing colorable (resp., $(2,2,2,2)$-packing colorable).
\end{proposition}
\begin{proof}
The graph $\mathcal{G}_2$ is obviously not $(1,2,2)$-packing colorable. Suppose to the contrary that the graph $\mathcal{G}_6$ has a $(2,2,2,2)$-packing coloring. Let $x,y,z$ (resp., $x',y',z'$) be the vertices of the first (resp., second) triangle. Let $u$ be the vertex having a common neighbor on the two triangles and $u_1$ (resp., $u_2$) be the other vertex having a neighbor in the first (resp., second) triangle. As each triangle contains three colors, $u$ and $u_1$ (resp., $u$ and $u_2$) have the same color. Thus, $u_1$ and $u_2$ have the same color, a contradiction.
\end{proof}\\

We proved that a $1$-saturated subcubic graph $G$ such that $g_3(G)=3$ is $(1,2,2,3)$-packing colorable. The result is sharp in several senses: there are such graphs that are not $(1,2,2)$-packing colorable (Proposition \ref{prop 3}), there are such graphs that are not $(2,2,2,2)$-packing colorable (Proposition \ref{prop 3}), and there is a $2$-saturated subcubic graph $G$ such that $g_3(G)=3$ that is not $(1,2,2,3)$-packing colorable (Proposition \ref{prop 4}).

\begin{proposition}\label{prop 4}
    The graph $\mathcal{G}_{10}$ is not $(1,2,2,3)$-packing colorable.
\end{proposition}
\begin{proof}
On the contrary, suppose that $\mathcal{G}_{10}$ has $(1,2,2,3)$-packing coloring. Note that the diameter of $\mathcal{G}_{10}$ is $3$, then at most one vertex can be colored by $3$. If three of the cycle of order $6$ are colored by $1$, then no more vertices are colored by $1$. In this case, at most four of the remaining vertices can be colored by the two colors $2$, a contradiction. If at most two vertices of the cycle of order $6$ are colored by $1$, then at most five of the remaining vertices can be colored by the two colors $2$ and the color $1$, a contradiction.
\end{proof}\\

We proved that a $1$-saturated subcubic graph $G$ such that $g_3(G)=3$ is $(2,2,2,2,3)$-packing colorable. The result is sharp in the sense that there are such graphs that are not $(2,2,2,2)$-packing colorable (Proposition \ref{prop 3}).

In addition, we pose the following problems.

\begin{problem}
    Is every $1$-saturated subcubic graph such that $g_3(G)=3$ $(1,2,3,3)$-packing colorable? $(1,2,2,4)$-packing colorable? $(2,2,2,2,4)$-packing colorable?
\end{problem}

We proved that a $1$-saturated subcubic graph $G$ such that $g_3(G)=3$ has $\chi_{\rho}(G)\leq 6$. We did not find such a graph $G$ such that $\chi_{\rho}(G)\geq 6$. Thus, we conjecture the following.

\begin{conjecture}
        If $G$ is a $1$-saturated subcubic graph such that $g_3(G)=3$, then $\chi_{\rho} (G)\leq 5$.
\end{conjecture}

We proved that every $1$-saturated subcubic graph $G$ such that $g_3(G)\leq 4$ is $(1,2,2,2)$-packing colorable. The result is sharp in several senses: there are $1$-saturated subcubic graphs such that $g_3(G)\leq 4$ that are not $(1,2,2)$-packing colorable (Proposition \ref{prop 5}), there are $1$-saturated subcubic graphs such that $g_3(G)\geq 5$ that are not $(1,2,2,2)$-packing colorable (Proposition \ref{prop 6}), and there is a $2$-saturated subcubic graph $G$ such that $g_3(G)\leq 4$ that is not $(1,2,2,2)$-packing colorable (Proposition \ref{prop 7}).

\begin{proposition}\label{prop 5}
    The graph $\mathcal{G}_7$ is not $(1,2,2)$-packing colorable.
\end{proposition}
\begin{proof}
On the contrary, suppose that $\mathcal{G}_7$ has a $(1,2,2)$-packing coloring. At most three vertices ar colored by $1$. If three vertices ar colored by $1$, then the remaining vertices are at a distance $2$ from each other. Thus, they cannot be colored using only two colors $2$, a contradiction. Otherwise, at most two vertices are colored by $1$, then at most two of the remaining vertices can be colored by the same color $2$, say $2_a$. The remaining vertices are at a distance $2$ from each other, then they cannot be colored by $2_b$, a contradiction.
\end{proof}

\begin{figure}[h!]
\centering
\begin{tikzpicture}[
    node/.style={circle, draw, fill=gray!10, minimum size=1mm},
    edge/.style={thick},
    green edge/.style={thick, draw=green},
    red edge/.style={thick, draw=red},
    black edge/.style={thick, draw=black},
    dotted edge/.style={thick, dotted},
    weight/.style={font=\bfseries\small}
]

\node[node] (1) at (0,0) {};
\node[node] (2) at (1,0) {};
\node[node] (3) at (2,0) {};
\node[node] (4) at (2,1) {};
\node[node] (5) at (1,1) {};
\node[node] (6) at (0,1) {};

\draw (1) -- (2);
\draw (2) -- (3);
\draw (3) -- (4);
\draw (4) -- (5);
\draw (5) -- (6);
\draw (6) -- (1);
\draw (2) -- (5);

\node at (1, -1) {$\mathcal{G}_7$};


\node[node] (x1) at (4,0) {};
\node[node] (x2) at (5,0) {};
\node[node] (x3) at (6,0) {};
\node[node] (x4) at (6,1) {};
\node[node] (x5) at (6,2) {};
\node[node] (x6) at (5,2) {};
\node[node] (x7) at (4,2) {};
\node[node] (x8) at (4,1) {};

\draw (x1) -- (x2);
\draw (x2) -- (x3);
\draw (x3) -- (x4);
\draw (x4) -- (x5);
\draw (x5) -- (x6);
\draw (x6) -- (x7);
\draw (x7) -- (x8);
\draw (x8) -- (x1);
\draw (x1) -- (x5);
\draw (x3) -- (x7);

\node at (5, -1) {$\mathcal{G}_8$};


\node[node] (y1) at (8.5,0) {};
\node[node] (y2) at (9.5,0) {};
\node[node] (y3) at (10,1) {};
\node[node] (y4) at (9,1) {};
\node[node] (y5) at (8,1) {};
\node[node] (y6) at (9,2) {};
\node[node] (y7) at (11,1) {};

\draw (y1) -- (y2);
\draw (y2) -- (y3);
\draw (y3) -- (y4);
\draw (y4) -- (y5);
\draw (y5) -- (y1);
\draw (y5) -- (y6);
\draw (y3) -- (y6);
\draw (y6) -- (y7);
\draw (y2) -- (y7);

\node at (9, -1) {$\mathcal{G}_9$};

\end{tikzpicture}
\caption{The graph $\mathcal{G}_7$: a $1$-saturated subcubic graph such that $g_3(\mathcal{G}_7)=4$ that is not $(1,2,2)$-packing colorable.\\
The graph $\mathcal{G}_8$: a $1$-saturated subcubic graph such that $g_3(\mathcal{G}_8)=5$ that is not $(1,2,2,2)$-packing colorable.\\
The graph $\mathcal{G}_9$: a $2$-saturated subcubic graph such that $g_3(\mathcal{G}_9)=4$ that is neither $(2,2,2,2,2)$-packing colorable nor $(1,2,2,2)$-packing colorable.}\label{figure 4}
\end{figure}
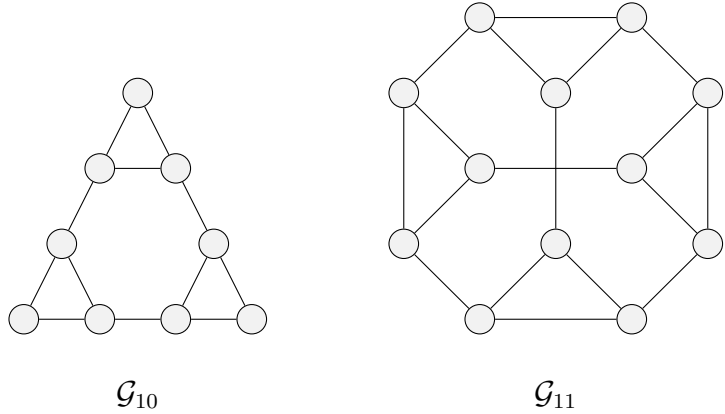
\begin{figure}[h!]
\centering
\begin{tikzpicture}[
    node/.style={circle, draw, fill=gray!10, minimum size=1mm},
    edge/.style={thick},
    green edge/.style={thick, draw=green},
    red edge/.style={thick, draw=red},
    black edge/.style={thick, draw=black},
    dotted edge/.style={thick, dotted},
    weight/.style={font=\bfseries\small}
]

\node[node] (1) at (0,0) {};
\node[node] (2) at (1,0) {};
\node[node] (3) at (2,0) {};
\node[node] (4) at (3,0) {};
\node[node] (5) at (2.5,1) {};
\node[node] (6) at (2,2) {};
\node[node] (7) at (1.5,3) {};
\node[node] (8) at (1,2) {};
\node[node] (9) at (0.5,1) {};

\draw (1) -- (2);
\draw (2) -- (3);
\draw (3) -- (4);
\draw (4) -- (5);
\draw (5) -- (6);
\draw (6) -- (7);
\draw (7) -- (8);
\draw (8) -- (9);
\draw (9) -- (1);
\draw (9) -- (2);
\draw (3) -- (5);
\draw (6) -- (8);

\node at (1.5, -1) {$\mathcal{G}_{10}$};


\node[node] (1) at (6,0) {};
\node[node] (2) at (8,0) {};
\node[node] (3) at (9,1) {};
\node[node] (4) at (9,3) {};
\node[node] (5) at (8,4) {};
\node[node] (6) at (6,4) {};
\node[node] (7) at (5,3) {};
\node[node] (8) at (5,1) {};

\node[node] (9) at (7,1) {};
\node[node] (10) at (8,2) {};
\node[node] (11) at (7,3) {};
\node[node] (12) at (6,2) {};

\draw (1) -- (2);
\draw (2) -- (3);
\draw (3) -- (4);
\draw (4) -- (5);
\draw (5) -- (6);
\draw (6) -- (7);
\draw (7) -- (8);
\draw (8) -- (1);
\draw (1) -- (9);
\draw (2) -- (9);
\draw (3) -- (10);
\draw (4) -- (10);
\draw (5) -- (11);
\draw (6) -- (11);
\draw (7) -- (12);
\draw (8) -- (12);
\draw (9) -- (11);
\draw (10) -- (12);

\node at (7, -1) {$\mathcal{G}_{11}$};

\end{tikzpicture}
\caption{The graph $\mathcal{G}_{10}$: a $2$-saturated subcubic graph such that $g_3(\mathcal{G}_{10})=3$ that is neither $(1,1,3,3)$-packing colorable nor $(1,2,2,3)$-packing colorable.\\
The graph $\mathcal{G}_{11}$: a $(3,3)$-saturated subcubic graph such that $g_3(\mathcal{G}_{11})=3$ that is not $(1,1,3,3,3)$-packing colorable.}\label{figure 5}
\end{figure}

\begin{proposition}\label{prop 6}
    The graph $\mathcal{G}_8$ is not $(1,2,2,2)$-packing colorable.
\end{proposition}
\begin{proof}
On the contrary, suppose that $\mathcal{G}_8$ has a $(1,2,2,2)$-packing coloring. Then, at least two $3$-vertices are colored by two different colors $2$. Suppose that $x$ and $y$ are two non-adjacent $3$-vertices that are colored by $2_a$ and $2_b$, respectively. The remaining six vertices cannot be colored by $2_a$ or $2_b$. It is clear that they cannot be colored by $1$ and $2_c$ only, a contradiction.
\end{proof}

\begin{proposition}\label{prop 7}
    The graph $\mathcal{G}_9$ is not $(1,2,2,2)$-packing colorable.
\end{proposition}
\begin{proof}
On the contrary, suppose that $\mathcal{G}_9$ has a $(1,2,2,2)$-packing coloring. As a cycle of order $5$ is not $(1,2,2)$-packing colorable, the vertices of the cycle of order $5$ that contains four $3$-vertices must uses the four colors $1,2_a,2_b,2_c$. The two remaining vertices are at a distance of at most $2$ from the vertices of the mentioned cycle. Then, they are colored by the color $1$. But at least one of their neighbors is colored by $1$, a contradiction.
\end{proof}\\

We proved that every $2$-saturated subcubic graph $G$ such that $g_3(G)=3$ is $(1,1,2)$-packing colorable. The result is sharp in the sense that there are such graphs that are not $(1,1,3,3)$-packing colorable (Proposition \ref{prop 8}) and there are such graphs that are not $(1,2,2,3)$-packing colorable (Proposition \ref{prop 4}). Moreover, as we did not find a $2$-saturated subcubic graph $G$ such that $g_3(G)\geq 4$ that is not $(1,1,2)$-packing colorable, we conjecture the following.

\begin{conjecture}
    Every $2$-saturated subcubic graph is $(1,1,2)$-packing colorable.
\end{conjecture}

\begin{proposition}\label{prop 8}
    The graph $\mathcal{G}_{10}$ is not $(1,1,3,3)$-packing colorable.
\end{proposition}
\begin{proof}
The diameter of $\mathcal{G}_{10}$ is $3$, then at most one vertex can be colored by $3_a$ (resp., $3_b$). But there are three disjoint triangles that must contain a vertex of color $3$. Thus, $\mathcal{G}_{10}$ does not have a $(1,1,3,3)$-packing coloring.
\end{proof}\\

We proved that every $2$-saturated subcubic graph $G$ such that $g_3(G)=3$ is $(1,2,2,2)$-packing colorable. The result is sharp in several senses: there are such graphs that are not $(1,2,2,3)$-packing colorable (Proposition \ref{prop 4}), there are such graphs that are not $(2,2,2,2)$-packing colorable (Proposition \ref{prop 1}), and there is a $2$-saturated subcubic graph $G$ such that $g_3(G)=4$ that is not $(1,2,2,2)$-packing colorable (Proposition \ref{prop 7}). In addition, we did not find a $2$-saturated subcubic graph $G$ such that $g_3(G)=4$ that is not $(1,2,2,2,2)$-packing colorable. Thus, we conjecture the following.

\begin{conjecture}
    Every $2$-saturated subcubic graph is $(1,2,2,2,2)$-packing colorable.
\end{conjecture}

We proved that every $2$-saturated subcubic graph $G$ such that $g_3(G)=3$ is $(2,2,2,2,2)$-packing colorable. The result is sharp in several senses: there are such graphs that are not $(2,2,2,2)$-packing colorable (Proposition \ref{prop 1}) and there is a $2$-saturated subcubic graph $G$ such that $g_4(G)=4$ that is not $(2,2,2,2,2)$-packing colorable (Proposition \ref{prop 9}). Moreover, we did not find a $2$-saturated subcubic graph $G$ such that $g_3(G)=3$ that is not $(2,2,2,2,3)$-packing colorable. Thus, we conjecture the following.

\begin{conjecture}
    Every $2$-saturated subcubic graph $G$ such that $g_3(G)=3$ is $(2,2,2,2,3)$-packing colorable.
\end{conjecture}

\begin{proposition}\label{prop 9}
    The graph $\mathcal{G}_9$ is not $(2,2,2,2,2)$-packing colorable.
\end{proposition}
\begin{proof}
On the contrary, suppose that $\mathcal{G}_9$ has a $(2,2,2,2,2)$-packing coloring. A cycle of order $5$ must be colored by the five colors $2$. The two remaining vertices are at a distance of at most $2$ from any vertex of the mentioned cycle, a contradiction.
\end{proof}\\

We proved that every $(3,0)$-saturated subcubic graph $G$ such that $g_3(G)=3$ is $(1,1,2,4)$-packing colorable. The result is sharp in the sense that there are such graphs that are not $(1,1,3,3)$-packing colorable (Proposition \ref{prop 8}). In addition, we did not find such a graph that is not $(1,1,2)$-packing colorable. So, we conjecture the following.

\begin{conjecture}
    Every $(3,0)$-saturated subcubic graph $G$ such that $g_3(G)=3$ is $(1,1,2)$-packing colorable.
\end{conjecture}

We proved that every $(3,1)$-saturated subcubic graph $G$ such that $g_3(G)=3$ is $(1,1,2,3)$-packing colorable. The result is sharp in the sense that there are such graphs that are not $(1,1,3,3)$-packing colorable (Proposition \ref{prop 8}).

We proved that every $(3,2)$-saturated subcubic graph $G$ such that $g_3(G)=3$ is $(1,1,3,3,3)$-packing colorable. The result is sharp in several sense: there are such graphs that are not $(1,1,3,3)$-packing colorable (Proposition \ref{prop 8}) and there is a $(3,3)$-saturated subcubic graph $G$ that is not $(1,1,3,3,3)$-packing colorable (Proposition \ref{prop 10}).

\begin{proposition}\label{prop 10}
    The graph $\mathcal{G}_{11}$ is not $(1,1,3,3,3)$-packing colorable.
\end{proposition}
\begin{proof}
The diameter of $\mathcal{G}_{11}$ is $3$, then at most one vertex can be colored by $3_a$ (resp., $3_b, 3_c$). But there are four disjoint triangles that must contain a vertex of color $3$. Thus, $\mathcal{G}_{11}$ does not have a $(1,1,3,3,3)$-packing coloring.
\end{proof}\\

Although we found a $(3,3)$-saturated subcubic graph $G$ such that $g_3(G)=3$ that is not $(1,1,3,3,3)$-packing colorable, we think it may be the unique counter-example. Thus, we conjecture the following.

\begin{conjecture}
    Every claw-free subcubic graph, except $\mathcal{G}_{11}$, is $(1,1,3,3,3)$-packing colorable.
\end{conjecture}

We also conjecture the following.

\begin{conjecture}
    Every claw-free subcubic graph is $(1,1,2,3)$-packing colorable.
\end{conjecture}
\begin{conjecture}
    Every claw-free subcubic graph is $(2,2,2,2,2,2,2)$-packing colorable.
\end{conjecture}

\section*{Acknowledgments}
This work was partially supported by a grant from the IMU-CDC and Simons Foundation.

\end{document}